\documentclass[11pt,letterpaper]{amsart}

\usepackage{graphicx}
\usepackage{amssymb}
\usepackage{amsthm}
\usepackage{amsmath}
\usepackage{stmaryrd}
\usepackage{cite}
\usepackage[mathscr]{eucal}
\usepackage{hyperref}
\usepackage[margin=1.7in]{geometry}
\usepackage{comment}
\usepackage{caption}
\usepackage{subcaption}
\usepackage[permil]{overpic}
\usepackage{url}
\usepackage{placeins}

\title[A Thorpe trick for the Bochner technique]{A Thorpe trick for the Bochner technique}
\author[Matthias Wink]{Matthias Wink}
\address{Department of Mathematics, University of California, Santa Barbara, South Hall 6607, Santa Barbara, CA 93106, USA}
\email{wink@math.ucsb.edu}


\makeatletter
\@namedef{subjclassname@2020}{
\textup{2020} Mathematics Subject Classification}
\makeatother


\begin{document}
\newcommand{\Div}{\operatorname{div}}
\newcommand{\Hol} {\operatorname{Hol}}
\newcommand{\diam} {\operatorname{diam}}
\newcommand{\Scal} {\operatorname{Scal}}
\newcommand{\scal} {\operatorname{scal}}
\newcommand{\Ric} {\operatorname{Ric}}
\newcommand{\Hess} {\operatorname{Hess}}
\newcommand{\grad} {\operatorname{grad}}
\newcommand{\Sect} {\operatorname{Sect}}
\newcommand{\Rm} {\operatorname{Rm}}
\newcommand{ \Rmzero } {\mathring{\Rm}}
\newcommand{\Rc} {\operatorname{Rc}}
\newcommand{\Curv} {S_{B}^{2}\left( \mathfrak{so}(n) \right) }
\newcommand{ \tr } {\operatorname{tr}}
\newcommand{ \id } {\operatorname{id}}
\newcommand{ \Riczero } {\mathring{\Ric}}
\newcommand{ \ad } {\operatorname{ad}}
\newcommand{ \Ad } {\operatorname{Ad}}
\newcommand{ \dist } {\operatorname{dist}}
\newcommand{ \rank } {\operatorname{rank}}
\newcommand{\Vol}{\operatorname{Vol}}
\newcommand{\dVol}{\operatorname{dVol}}
\newcommand{ \zitieren }[1]{ \hspace{-3mm} \cite{#1}}
\newcommand{ \pr }{\operatorname{pr}}
\newcommand{\diag}{\operatorname{diag}}
\newcommand{\Lagr}{\mathcal{L}}
\newcommand{\av}{\operatorname{av}}
\newcommand{ \floor }[1]{ \lfloor #1 \rfloor }
\newcommand{ \ceil }[1]{ \lceil #1 \rceil }
\newcommand{\Sym} {\operatorname{Sym}}
\newcommand{\bcirc}{ \ \bar{\circ} \ }
\newcommand{\conj}[1]{ \overline{ #1 } }
\newcommand{\sign}[1]{\operatorname{sign}(#1)}
\newcommand{\cone}{\operatorname{cone}}
\newcommand{\pbd}{\varphi_{bar}^{\delta}}

\newtheorem{theorem}{Theorem}[section]
\newtheorem{definition}[theorem]{Definition}
\newtheorem{example}[theorem]{Example}
\newtheorem{remark}[theorem]{Remark}
\newtheorem{lemma}[theorem]{Lemma}
\newtheorem{proposition}[theorem]{Proposition}
\newtheorem{corollary}[theorem]{Corollary}
\newtheorem{assumption}[theorem]{Assumption}
\newtheorem{acknowledgment}[theorem]{Acknowledgment}
\newtheorem{DefAndLemma}[theorem]{Definition and lemma}

\newcommand{\R}{\mathbb{R}}
\newcommand{\N}{\mathbb{N}}
\newcommand{\Z}{\mathbb{Z}}
\newcommand{\Q}{\mathbb{Q}}
\newcommand{\C}{\mathbb{C}}
\newcommand{\F}{\mathbb{F}}
\newcommand{\X}{\mathcal{X}}
\newcommand{\D}{\mathcal{D}}
\newcommand{\Cont}{\mathcal{C}}

\renewcommand{\labelenumi}{(\alph{enumi})}
\newtheorem{maintheorem}{Theorem}[]
\renewcommand*{\themaintheorem}{\Alph{maintheorem}}
\newtheorem*{theorem*}{Theorem}
\newtheorem*{corollary*}{Corollary}
\newtheorem*{remark*}{Remark}
\newtheorem*{example*}{Example}
\newtheorem*{question*}{Question}
\newtheorem*{definition*}{Definition}
\newtheorem{conjecture}{Conjecture}
\newtheorem*{conjecture*}{Conjecture}
\renewcommand*{\theconjecture}{\Alph{conjecture}}

\begin{abstract}
We prove that for $n \geq 6$ there exists $\varepsilon (n)>0$ such that every closed orientable Riemannian manifold with $\left( \ceil{\frac{n}{2}}+ \varepsilon\right)$-positive curvature operator is a real homology sphere. For $n=6$ we show that the same result holds for manifolds with $4$-positive curvature operators. 
\end{abstract}

\maketitle

\section*{Introduction}

Vanishing and sphere theorems for closed manifolds have a long history in Riemannian Geometry. In his seminal paper \cite{BochnerVectorFieldsAndRic}, 
Bochner proved that the first Betti number of Riemannian manifolds with positive Ricci curvature vanishes. Building up on Bochner's technique, Berger \cite{BergerTwoFormsVanishCurvOperator} and D. Meyer \cite{DMeyerCurvOpPos} proved that manifolds with positive curvature operators are real homology spheres. With Ricci flow techniques, Hamilton \cite{Hamilton3DimRF,Hamilton4DimRFposCurvOp}, Chen \cite{ChenQuarterPinching} and B\"ohm-Wilking \cite{BW2} improved these results and showed that manifolds with $2$-positive curvature operators are space forms. In the context of the proof of the differentiable sphere theorem, further generalizations were obtained by Brendle-Schoen \cite{BrendleSchoenSphereTheorem} and Brendle \cite{BrendleConvergenceInHigherDimensions}.

For manifolds with positive isotropic curvature, Micallef-Wang \cite{MicallefWangNIC} proved that in the even-dimensional case the second Betti number vanishes, and in the simply connected case the manifold is a homotopy sphere due to Micallef-Moore \cite{MicallefMoorePIC}. Classification results were obtained by Hamilton \cite{HamiltonFourPIC}, Chen-Zhu \cite{ChenZhuRFwithSurgeryFourPIC} and Chen-Tang-Zhu \cite{ChenTangZhuClassFourPIC} in dimension $n=4$, Brendle \cite{BrendleRFwithSurgeryPIC} and Huang \cite{HuangManifoldsPIC}  for $n \geq 12,$ Chen \cite{ChenPIC} for $n=9, 10, 11,$ Cho \cite{ChoPIC} for $n=7,8$ and Chow-Wang \cite{ChowWangMinimalTwoSpheresPIC} for $n=5,6.$

Vanishing and sphere theorems for the curvature operator of the second kind were established by Cao-Gursky-Tran in  
\cite{CaoGurskyTranNishikawaConjecture}, together with Nienhaus and Petersen in \cite{NienhausPetersenWinkNewCurvatureOperatorSecondKind} or X. Li in \cite{LiSphereTheoremsSecondKind}.

In \cite{PetersenWinkNewCurvatureConditionsBochner}, Petersen and the author proved that manifolds with $\ceil{\frac{n}{2}}$-positive curvature operators are real homology spheres. The first main theorem of this paper improves upon this result. 

\begin{maintheorem}
\label{GeneralHomologySphereTheorem}
    For every $n \geq 6$ there exists $\varepsilon(n)>0$ such that every closed orientable Riemannian manifold with $\left( \ceil{\frac{n}{2}} + \varepsilon \right)$-positive curvature operator is a real homology sphere. 
\end{maintheorem}

The curvature operator $\mathcal{R} \colon \Lambda^2TM \to \Lambda^2TM$ is called $k$-nonnegative if its eigenvalues $\lambda_1 \leq \ldots \leq \lambda_{\binom{n}{2}}$ satisfy 
\begin{align*}
    \lambda_1 + \ldots + \lambda_{\floor{k}} + (k-\floor{k}) \lambda_{\floor{k}+1} \geq 0
\end{align*}
and $k$-positive if the inequality is strict. 

Theorem \ref{GeneralHomologySphereTheorem} is a consequence of a more general vanishing result for Betti numbers, generalizing \cite[Theorem A]{PetersenWinkNewCurvatureConditionsBochner}. Specifically we have

\begin{maintheorem}
\label{EpsilonImprovementManifolds}
    For $3 \leq p \leq \frac{n}{2}$ there exists $\varepsilon(p,n)>0$ with the following property. If $(M,g)$ is a closed Riemannian manifold with $(n-p+\varepsilon)$-positive curvature operator, then the Betti numbers $b_p(M)$ and $b_{n-p}(M)$ vanish.
\end{maintheorem}

It is possible to estimate $\varepsilon(p,n)>0$ explicitly, see Remark \ref{EffectiveEstimateRemark}. In particular, one may choose $\varepsilon>0$ independent of $n.$ For $p=3$ and $n=6,$ we can obtain the optimal estimate that our techniques allow, compare Example \ref{ExampleSharpCandidate}, and prove

\begin{maintheorem}
\label{SixDimHomologySpheres}
     Let $(M,g)$ be a $6$-dimensional closed orientable Riemannian manifold. If the curvature operator of $(M,g)$ is $4$-positive, then $M$ is a real homology sphere.
\end{maintheorem}

Theorems \ref{GeneralHomologySphereTheorem} -- \ref{SixDimHomologySpheres} rely on the Bochner technique. In particular, our methods also imply vanishing and estimation results. The rigidity case for manifolds with nonnegative curvature operators is due to Gallot-Meyer \cite{GallotMeyerCurvOperatorAndForms} and the estimation theorem for manifolds with curvature operator bounded below is due to Gallot \cite{GallotSobolevEstimates}, building up on work of P. Li \cite{LiSobolevConstant}. Theorems \ref{GeneralRigidityTheorem} and \ref{TheoremThirdBettiNumber} below generalize \cite[Theorems A -- C]{PetersenWinkNewCurvatureConditionsBochner}.

\begin{maintheorem}
    \label{GeneralRigidityTheorem}
    For $3 \leq p \leq \frac{n}{2}$ there exists $\varepsilon(p,n)>0$ with the following property. Let $(M,g)$ be a closed $n$-dimensional Riemannian manifold and let $\lambda_1 \leq  \ldots \leq \lambda_{\binom{n}{2}}$ denote the eigenvalues of the curvature operator $\mathcal{R}$.
   \begin{enumerate}
       \item If $\mathcal{R}$ is $(n-p+\varepsilon)$-positive, then $b_{p}(M)=b_{n-p}(M)=0.$
       \item If $\mathcal{R}$ is $(n-p+\varepsilon)$-nonnegative, then all harmonic $p$-forms are parallel.
       \item For $\kappa \leq 0$ and $D>0$ there is $C(n,\kappa D^2)>0$ such that if $\diam(M) < D$ and $\lambda_1 + \ldots + \lambda_{n-p} + \varepsilon \lambda_{n-p+1} \geq (n-p+\varepsilon) \kappa,$  then
       \begin{align*}
           b_p(M) \leq \binom{n}{p} \exp\left( C \sqrt{-\kappa D^2 p(n-p)} \right).
       \end{align*}
   \end{enumerate}
\end{maintheorem}

For $3$-forms in dimension $n=6$ it is possible to use a normal form to refine the techniques used in the proof of Theorem \ref{GeneralRigidityTheorem}. In particular, Theorem \ref{SixDimHomologySpheres} is a consequence of the following more general result and \cite{BochnerVectorFieldsAndRic,PetersenWinkNewCurvatureConditionsBochner}.

\begin{maintheorem}
\label{TheoremThirdBettiNumber}
    Let $(M,g)$ be a $6$-dimensional closed Riemannian manifold and let $\lambda_1 \leq  \ldots \leq \lambda_{15}$ denote the eigenvalues of the curvature operator $\mathcal{R}$.
   \begin{enumerate}
       \item If $\mathcal{R}$ is $\frac{9}{2}$-positive, then $b_{3}(M)=0.$
       \item If $\mathcal{R}$ is $\frac{9}{2}$-nonnegative, then all harmonic $3$-forms are parallel.
       \item For $\kappa \leq 0$ and $D>0$ there is $C(\kappa D^2)>0$ such that if $\diam(M) < D$ and $\lambda_1 + \ldots + \lambda_4 + \frac{1}{2} \lambda_5 \geq \frac{9}{2} \kappa,$  then $b_3(M) \leq 20 \exp\left( C \sqrt{-\kappa} D \right).$
   \end{enumerate}
\end{maintheorem}

The proofs of Theorems \ref{GeneralHomologySphereTheorem} --  \ref{TheoremThirdBettiNumber} are based on a new tool in the Bochner technique, which we will describe in the following. If $\omega \in \Lambda^p T^{*}M$ is a harmonic form, then 
\begin{align*}
    \Delta \frac{1}{2} | \omega|^2 = | \nabla \omega|^2 + g(\Ric_L(\omega),\omega).
\end{align*}
In particular, the maximum principle implies that all harmonic forms are parallel provided $g(\Ric_L( \omega), \omega) \geq 0.$ Poor \cite{PoorHolonomyProofPosCurvOperatorThm} noticed that the curvature term of the Lichnerowicz Laplacian can be written in terms of the curvature operator. Explicitly, if $\{ \Xi_{\alpha} \}$ is an orthonormal eigenbasis for $\mathcal{R}$ with corresponding eigenvalues $\{ \lambda_{\alpha} \},$ then 
\begin{align*}
    g(\Ric_L(\omega),\omega) = \sum\nolimits_{\alpha} \lambda_{\alpha} | \Xi_{\alpha} \omega |^2,
\end{align*}
where $\Xi \omega = \Xi \cdot \omega$ is the Lie algebra action of $\Lambda^2TM \cong \mathfrak{so}(TM)$ on $\Lambda^p T^*M.$ In \cite{PetersenWinkNewCurvatureConditionsBochner}, Petersen and the author used this identity to obtain vanishing, rigidity and estimation results. We note that their estimates work for every self-adjoint operator $R \in \Sym^2(\Lambda^2TM)$ and improve upon their techniques by introducing a Thorpe trick to the Bochner technique. Specifically, if we define 
\begin{align*}
    \mathcal{A}_{\omega} \colon \Lambda^2TM \to \Lambda^2TM, \ g(  \mathcal{A}_{\omega} (\alpha), \beta) = g( \alpha \cdot \omega, \beta \cdot \omega),
\end{align*}
then
\begin{align*}
    g(\Ric_L(\omega),\omega) = \tr( \mathcal{R} \mathcal{A}_{\omega}) = \tr( \mathcal{R} (\mathcal{A}_{\omega}+\hat{\eta}))
\end{align*}
for every $\hat{\eta} \in\widehat{ \Lambda^4T^*M}$ since the space of algebraic curvature operators is orthogonal to operators induced by $4$-forms. We may therefore change $\mathcal{A}_{\omega}$ by a $4$-form operator without changing the curvature term in the Lichnerowicz Laplacian. The key idea is to replace $\mathcal{A}_{\omega}$ with an operator that has better nonnegativity properties. Similar ideas were used in the context of sectional curvature in \cite{ThorpeCurvatureTensorPosCurvedFourMf, PuettmannOptimalPinchingConstants, GroveVerdianiZillerExoticT1S4Pos, BettiolMendesStronglyPositiveCurvature}.

In Theorem \ref{SharpnessCriterion} we provide a sharp criterion characterizing nonnegativity of $\Ric_L$ on $p$-forms in terms of eigenvalues of the curvature operator, generalizing the approach in \cite{PetersenWinkNewCurvatureConditionsBochner}. In particular, we show that $\Ric_L \geq 0$ on $\Lambda^pT^*M$ for every $k$-nonnegative curvature tensor if and only if for every $\omega \in \Lambda^p T^*M$ there exists $\eta \in \Lambda^4T^*M$ such that 
\begin{align*}
      0 \leq \mathcal{A}_{\omega} + \hat{\eta} \leq \frac{\tr \mathcal{A}_{\omega}}{k}  \id_{\Lambda^2TM} = \frac{p(n-p)}{k} |\omega|^2  \id_{\Lambda^2TM}.
\end{align*}

For $p \leq \frac{n}{2},$ the estimate $| L \omega|^2 \leq p |L|^2 |\omega|^2$ from \cite[Lemma 2.2]{PetersenWinkNewCurvatureConditionsBochner} shows that one can take $\eta=0$ and $k=n-p.$ By analyzing the action operator $\mathcal{A}_{\omega}$ for forms $\omega \in \Lambda^pT^*M$ with $|L\omega|^2 = p |L|^2 |\omega|^2,$ we deduce in Theorem \ref{EpsilonImprovement} that for $3 \leq p \leq \frac{n}{2}$ and every $\omega \in \Lambda^pT^*M$ there exists $\eta \in \Lambda^4T^*M$ such that 
\begin{align*}
     0 \leq \mathcal{A}_{\omega} + \hat{\eta} \leq (p - \varepsilon) |\omega|^2 \id_{\Lambda^2TM}.
\end{align*}
Remark \ref{EffectiveEstimateRemark} provides a quantitative estimate on $\varepsilon>0$ and in particular shows that $\varepsilon>0$ may be chosen independent of $n.$ Furthermore, no such $\varepsilon>0$ can be obtained for $p=1,2.$

Moreover, using a normal form for $3$-forms in dimension $n=6,$ see Theorem \ref{NormalForm}, we prove in Theorem \ref{ThorpeTrick36} that for every $\omega \in \Lambda^3T^*M^6$ there exists $\eta \in \Lambda^4T^*M^6$ such that
\begin{align*}
     0 \leq \mathcal{A}_{\omega} + \hat{\eta} \leq 2 |\omega|^2 \id_{\Lambda^2TM^6}.
\end{align*}
Together with the nonnegativity criterion in Theorem \ref{SharpnessCriterion} this implies Theorems \ref{GeneralRigidityTheorem} and \ref{TheoremThirdBettiNumber}. Example \ref{ExampleSharpCandidate} shows that the factor $2$ obtained in Theorem \ref{ThorpeTrick36} is optimal.

More generally, we conjecture that an analogous theorem holds in higher dimensions. Specifically, we ask for
\begin{conjecture}
\label{ConjectureControllingLichnerowicz}
    Let $(V,g)$ be a Euclidean vector space. For every $\omega \in \Lambda^p V^{*}$ there exists $\eta \in \Lambda^4V^*$ such that
    \begin{align*}
        0 \leq \mathcal{A}_{\omega} + \hat{\eta} \leq 2 |\omega|^2 \id_{\Lambda^2V}.
    \end{align*}
\end{conjecture}
Example \ref{ExampleSharpCandidate} shows that the factor $2$ in Conjecture \ref{ConjectureControllingLichnerowicz} cannot be improved. Together with Theorem \ref{SharpnessCriterion}, Conjecture \ref{ConjectureControllingLichnerowicz} would imply 
\begin{conjecture}
\label{RiemannianConjecture}
    Let $(M,g)$ be a closed $n$-dimensional Riemannian manifold. If the curvature operator of $(M,g)$ is $\frac{p(n-p)}{2}$-positive, then the $p$-th Betti number vanishes, $b_p(M)=0,$ for $2 \leq  p \leq n-2.$

    In particular, if $(M,g)$ is orientable with $(n-2)$-positive curvature operator, then $M$ is a real homology sphere.
\end{conjecture}
Note that $S^1 \times S^{n-1}$ has $(n-1)$-nonnegative curvature operator but is neither simply connected nor admits a metric of positive Ricci curvature. Theorem \ref{TheoremThirdBettiNumber} proves Conjecture \ref{RiemannianConjecture} in dimension $n=6.$\vspace{2mm}

\textbf{AI disclosure.} The author used OpenAI’s ChatGPT as an assistive tool in the development of the paper. The author takes full responsibility for all mathematical arguments and the manuscript was entirely written by the author. 
\vspace{2mm}

\textbf{Acknowledgments.} The author's research is partially supported by travel grant SFI-MPS-TSM-00026032 funded by the Simons Foundation International and administered by the Simons Foundation.

\section{Preliminaries}

Let $(V,g)$ be an $n$-dimensional Euclidean vector space. The vector space of self-adjoint operators on $\Lambda^2V$ decomposes into orthogonal $O(V)$-invariant subspaces as
\begin{align*}
    \Sym^2(\Lambda^2V) = \Sym^2_B(\Lambda^2V) \oplus \widehat{\Lambda^4 V^{*}}
\end{align*}
where $\Sym^2_B(\Lambda^2V)$ is the space of algebraic curvature operators. Note that the operator $\hat{\eta}$ induced by a $4$-form $\eta \in \Lambda^4 V^{*}$ is given by
\begin{align*}
    g(\hat{\eta}(x \wedge y), z \wedge w) = \eta(x,y,z,w).
\end{align*}

We identify $\Lambda^2V$ with $\mathfrak{so}(V)$ by setting 
\begin{align*}
    (x \wedge y)z = g(x,z)y - g(y,z)x
\end{align*}
and use the convention that if $e_1, \ldots, e_n$ is an orthonormal basis for $V,$ then $\{ e_{i_1} \wedge \ldots \wedge e_{i_p} \}_{1 \leq i_1 < \ldots < i_p \leq n}$ is an orthonormal basis for $\Lambda^pV.$ The identification of $\Lambda^2V$ with $\mathfrak{so}(V)$ becomes an isometry if $\left\langle A, B \right\rangle = - \frac{1}{2} \tr (AB)$ is the inner product on $\mathfrak{so}(V).$ 

The orthogonal group $O(V)$ acts on $\Lambda^pV$ by
\begin{align*}
    O \cdot x_1 \wedge \ldots \wedge x_p = Ox_1 \wedge \ldots \wedge Ox_p.
\end{align*}
The induced action on $p$-forms $\Lambda^pV^{*}$ is given by
\begin{align*}
    (O \cdot \omega)(X_1, \ldots, X_p) = \omega(O^{-1}X_1, \ldots, O^{-1}X_p)
\end{align*}
and the corresponding Lie algebra action of $\mathfrak{so}(V)$ is 
\begin{align*}
    (\Xi \cdot \omega)(X_1, \ldots, X_p) = - \sum_{i=1}^p \omega(X_1, \ldots, \Xi X_i, \ldots, X_p).
\end{align*}
Following \cite{PetersenWinkNewCurvatureConditionsBochner}, we may also use the notation $\Xi\omega$ or $(\Xi)\omega$ for the Lie algebra action of $\Lambda^2V \cong \mathfrak{so}(V)$ on forms.

\section{The curvature term of the Lichnerowicz Laplacian}

\begin{definition}\label{ActionOperator}
    For $\omega \in \Lambda^p V^{*}$ define $\mathcal{A}_{\omega} \colon \Lambda^2V \to \Lambda^2V$ by
    \begin{align*}
      g(  \mathcal{A}_{\omega} (\alpha), \beta) = g( \alpha \cdot \omega, \beta \cdot \omega)
    \end{align*}
for $\alpha, \beta \in \Lambda^2V \cong \mathfrak{so}(V).$ We call $\mathcal{A}_{\omega}$ the {\em action operator} induced by $\omega.$
\end{definition}
Note that $\mathcal{A}_{\omega}$ is self-adjoint, nonnegative and satisfies
\begin{align} \label{ActionEquivariance}
    \mathcal{A}_{O \cdot \omega} (\alpha)= O \cdot \mathcal{A}_{\omega}\left(O^{-1} \cdot \alpha \right)
\end{align}
for all $O \in O(V)$ and $\alpha \in \Lambda^2V.$

Moreover, the estimates in \cite[Lemma 2.2 and Proposition 2.5]{PetersenWinkNewCurvatureConditionsBochner} say that 
\begin{align*}
    \tr \mathcal{A}_{\omega} = p(n-p) | \omega|^2 \text{ and } \ 0 \leq
    \mathcal{A}_{\omega}   \leq \min \{ p, n-p \}  | \omega |^2 \id_{\Lambda^2 V}.
\end{align*}

We define the bilinear operator $\mathcal{B}$ associated to $\mathcal{A}$ as follows. 

\begin{definition}
    For $\omega, \omega_1, \omega_2 \in \Lambda^p V^*$ define 
    \begin{enumerate}
        \item $\mathcal{D}_{\omega} \colon \Lambda^2V \to \Lambda^pV^*$ by $\mathcal{D}_{\omega}\alpha = \alpha \cdot \omega$ and 
        \item $\mathcal{B}_{\omega_1, \omega_2} \colon \Lambda^2V \to \Lambda^2V$ by $\mathcal{B}_{\omega_1, \omega_2} = \mathcal{D}_{\omega_1}^* \mathcal{D}_{\omega_2} + \mathcal{D}_{\omega_2}^* \mathcal{D}_{\omega_1}.$
    \end{enumerate}
\end{definition}

We have the following elementary properties.

\begin{proposition}
\label{PropertiesBilinearOperator}
If $\omega, \omega_1, \omega_2 \in \Lambda^p V^*$ and $\alpha, \beta \in \Lambda^2V,$ then
    \begin{enumerate}
        \item $|| \mathcal{D}_{\omega} ||^2_{\text{op}} \leq p |\omega|^2,$
        \item $\mathcal{B}_{\omega_1, \omega_2}$ is self-adjoint,
        \item $|| \mathcal{B}_{\omega_1, \omega_2} ||_{\text{op}} \leq 2p |\omega_1| | \omega_2|,$
        \item $g(\mathcal{B}_{\omega_1, \omega_2}\alpha, \beta)  = g( \alpha \cdot \omega_1, \beta \cdot \omega_{2}) + g( \beta \cdot \omega_{1}, \alpha \cdot \omega_2)$ 
        \item[] \hspace{21.4mm} $ = 2g( \alpha \cdot \omega_1, \beta \cdot \omega_{2}) -  g([\alpha, \beta] \cdot \omega_1, \omega_2).$
        \item[(e)] $\mathcal{A}_{\omega_1 + \omega_2}= \mathcal{A}_{\omega_1} + \mathcal{B}_{\omega_1,\omega_2} + \mathcal{A}_{\omega_2}.$ 
    \end{enumerate}
\end{proposition}
\begin{proof}
    (a) is the estimate \cite[Lemma 2.2]{PetersenWinkNewCurvatureConditionsBochner} as $|| \mathcal{D}_{\omega} ||_{\text{op}}^2 = \lambda_{\max}(\mathcal{D}_{\omega}^* \mathcal{D}_{\omega}) = \lambda_{\max}(\mathcal{A}_{\omega}) \leq p | \omega |^2.$ Furthermore, (b) follows from the definition, (c) directly from (a), and (d) follows from skew symmetry of $\alpha, \beta \in \Lambda^2V \cong \mathfrak{so}(V)$ and the fact that $\alpha \mapsto \alpha \cdot \omega$ is a Lie algebra action. (e) is clear from the definition.
\end{proof}

If $T$ is a $(0,k)$-tensor on $V$ and $R$ an algebraic curvature tensor, then set
\begin{align*}
    (\Ric_LT)(X_1, \ldots, X_k) = \sum_{i=1}^k \sum_{j=1}^n (R(X_i, e_j)T)(X_1, \ldots, e_j, \ldots, X_k),
\end{align*}
where $R(X,Y)$ is considered in $\mathfrak{so}(V).$ Building up on work of Poor \cite{PoorHolonomyProofPosCurvOperatorThm}, Petersen established in \cite[Lemma 9.3.3 and Lemma 9.4.3]{PetersenRiemGeom} that 
\begin{align*}
    g(\Ric_L(\omega), \omega) = \sum\nolimits_{\alpha} \lambda_{\alpha} | \Xi_{\alpha} \omega |^2,
\end{align*}
where $\{ \Xi_{\alpha} \}$ is an orthonormal eigenbasis of the curvature operator associated to $R$ with eigenvalues $\{ \lambda_{\alpha} \}$, i.e., $\mathcal{R}( \Xi_{\alpha}) = \lambda_{\alpha} \Xi_{\alpha}.$

Since algebraic curvature operators are orthogonal to $4$-forms, we obtain the following key observation immediately from Definition \ref{ActionOperator}.

\begin{proposition} \label{LichnerowiczAsTrace}
If $\omega \in \Lambda^p V^{*}$ and $\mathcal{R} \in \Sym^2_B(\Lambda^2V),$ then 
\begin{align*}
    g(\Ric_L(\omega),\omega) = \tr \left( \mathcal{R} \mathcal{A}_{\omega} \right) = \tr \left( \mathcal{R} (\mathcal{A}_{\omega} + \hat{\eta}) \right)
\end{align*}
for every $\eta \in \Lambda^4 V^{*}.$
\end{proposition}

Theorem \ref{SharpnessCriterion} provides a sharp criterion for when sums of eigenvalues of the curvature operator control the curvature term of the Lichnerowicz Laplacian. It is based on Proposition \ref{LichnerowiczAsTrace} and the following lemma.

\begin{lemma} \label{EigenvalueLemma}
    Let $(E,g)$ be a Euclidean vector space, $1 \leq k \leq \dim E$ and $\kappa \in \R.$    
    \begin{enumerate}
        \item Let $R,A \in \Sym^2(E)$ and suppose that $0 \leq A \leq \frac{\tr(A)}{k} \id_E$.
        
        If the eigenvalues of $R$ satisfy $\lambda_1 + \ldots + \lambda_{\floor{k}} + (k-\floor{k}) \lambda_{\floor{k}+1} \geq k \kappa$ then $\tr(RA) \geq \tr(A)\kappa.$
        
        If $R$ is $k$-positive and $\tr(A)>0,$ then $\tr(RA)>0.$
        \item If $A \in \Sym^2(E)$ satisfies $\tr(RA) \geq 0$ for every $R \in \Sym^2(E)$ with $\lambda_1 + \ldots + \lambda_{\floor{k}} + (k-\floor{k}) \lambda_{\floor{k}+1} \geq k \kappa,$ then $0 \leq A \leq \frac{\tr(A)}{k} \id_E.$
    \end{enumerate}
\end{lemma}
\begin{proof}
    (a) If $\{ \Xi_{\alpha} \}$ is an orthonormal eigenbasis for $R$ with corresponding eigenvalues $\{ \lambda_{\alpha} \},$ then
    \begin{align*}
        \tr(RA)  = & \ \sum_{\alpha} \lambda_{\alpha} g(A \Xi_{\alpha}, \Xi_{\alpha}) \\
        \geq & \ \sum_{\alpha=1}^{\floor{k}} \lambda_{\alpha} g(A \Xi_{\alpha}, \Xi_{\alpha}) + \lambda_{\floor{k}+1} \sum_{\alpha \geq \floor{k}+1} g(A \Xi_{\alpha}, \Xi_{\alpha}) \\
        = & \ \sum_{\alpha=1}^{\floor{k}} \left( \lambda_{\alpha} - \lambda_{\floor{k}+1} \right) g(A \Xi_{\alpha}, \Xi_{\alpha}) + \lambda_{\floor{k}+1} \tr(A) \\
        \geq & \ \frac{\tr(A)}{k} \sum_{\alpha=1}^{\floor{k}} \left( \lambda_{\alpha} - \lambda_{\floor{k}+1} \right) + \lambda_{\floor{k}+1} \tr(A) \\
        = & \ \frac{\tr(A)}{k} \left( \sum_{\alpha=1}^{\floor{k}} \lambda_{\alpha} + (k-\floor{k}) \lambda_{\floor{k}+1} \right)
    \end{align*}
    and the claim follows.

    (b) Suppose that $A(\Xi)=\lambda \Xi$ and $|\Xi|=1.$ First consider $\kappa \leq 0.$ Taking $R_1= \Xi \otimes g(\Xi, \cdot)$ implies $\lambda = \tr(R_1A) \geq 0,$ hence $A \geq 0.$ Set $R_2=\frac{1}{k-1} \id_E - \frac{k}{k-1} \Xi \otimes g(\Xi, \cdot).$ Then $R_2(\Xi)=-\Xi$ and all other eigenvalues are $\frac{1}{k-1}.$ Hence $R_2$ is $k$-nonnegative and therefore $0 \leq \tr(R_2A)=\frac{1}{k-1} \tr(A) - \frac{k}{k-1} \lambda = \frac{\tr(A)-k\lambda}{k-1}.$ In particular, $A \leq \frac{\tr(A)}{k} \id_E.$

    If $\kappa >0,$ consider $a R_1 + \kappa \id_E$ to conclude $a\lambda + \kappa \tr(A) \geq 0$ for every $a >0$ and hence $A \geq 0.$ Similarly, consider $a(k-1)R_2 + \kappa \id_E$ to deduce $\lambda \leq \frac{\tr(A)}{k} \frac{a+ \kappa}{a}$ for every $a >0.$
\end{proof}

Recall that for a subset $C \subseteq \Sym^2{(\Lambda^2V)}$ its dual cone is given by
\begin{align*}
    C^{*} = \{ A \in \Sym^2{(\Lambda^2V)} \ | \tr(RA) \geq 0 \text{ for all } R \in C \}.
\end{align*}
Consider the closed convex cone  
\begin{align*}
    S_k = \{ R \in \Sym^2(\Lambda^2V) \ | \ \lambda_1 + \ldots + \lambda_{\floor{k}} + (k-\floor{k}) \lambda_{\floor{k}+1} \geq 0 \}.
\end{align*}

\begin{proposition} \label{DualCones}
\label{ConeLemma}
(a) The dual cone of $S_k$ is
\begin{align*}
    S_k^* = \left\lbrace A \in \Sym^2( \Lambda^2V) \ | \ 0 \leq A \leq \frac{\tr(A)}{k} \id_{\Lambda^2V} \right\rbrace.
\end{align*}
(b) The dual cone of the cone of $k$-nonnegative algebraic curvature operators is
\begin{align*}
    \left( S_k \cap \Sym_B^2( \Lambda^2V) \right)^* = S_k^* + \widehat{\Lambda^4V^*}.
\end{align*}
\end{proposition}
\begin{proof}
    (a) This follows from Lemma \ref{EigenvalueLemma}. For (b), note that the curvature operator of the round sphere is in the intersection of the relative interiors of $S_k$ and $\Sym_B^2( \Lambda^2V),$ and hence \cite[Corollary 16.4.2]{RockafellarConvexAnalysis} applied to the closed convex cones $S_k$ and $\Sym_B^2(\Lambda^2V)$ shows 
    \begin{align*}
        \left( S_k \cap \Sym_B^2( \Lambda^2V) \right)^* = S_k^* + \left( \Sym_B^2( \Lambda^2V) \right)^* = S_k^* +  \widehat{\Lambda^4V^*}
    \end{align*}
as $\widehat{\Lambda^4V^*}$ is the orthogonal complement to $\Sym_B^2(\Lambda^2V)$ in $\Sym^2(\Lambda^2V).$
\end{proof}

\begin{theorem}\label{SharpnessCriterion}
    Let $\kappa \in \R.$ The following are equivalent:
    \begin{enumerate}
        \item Every $\mathcal{R} \in \Sym^2_B(\Lambda^2V)$ with $\lambda_1 + \ldots + \lambda_{\floor{k}} + (k-\floor{k}) \lambda_{\floor{k}+1} \geq k \kappa$ satisfies $g(\Ric_L(\omega),\omega) \geq \kappa p (n-p) |\omega|^2$ for all $\omega \in \Lambda^pV^{*}.$
        \item For every $\omega \in \Lambda^p V^{*}$ there exists $\eta \in \Lambda^4 V^{*}$ such that 
        \begin{align*}
            0 \leq \mathcal{A}_{\omega} + \hat{\eta} \leq \frac{\tr (\mathcal{A}_{\omega}+\hat{\eta})}{k} \id_{\Lambda^2V} = \frac{p(n-p)}{k} |\omega|^2  \id_{\Lambda^2V}.
        \end{align*}
    \end{enumerate}
\end{theorem}
\begin{proof}
(b) implies (a). Combining Proposition \ref{LichnerowiczAsTrace} and Lemma \ref{EigenvalueLemma} (a) shows that $g(\Ric_L(\omega), \omega) = \tr( \mathcal{R}(\mathcal{A}_{\omega} + \hat{\eta})) \geq \kappa \tr( \mathcal{A}_{\omega} + \hat{\eta})= \kappa \tr( \mathcal{A}_{\omega}) = \kappa p(n-p) |\omega|^2.$ 

(a) implies (b). Note that the operator $\mathcal{R} - \kappa \id$ is $k$-nonnegative and $\tr( (\mathcal{R}- \kappa \id) \mathcal{A}_{\omega}) = \tr( \mathcal{R} \mathcal{A}_\omega) - \kappa p(n-p)|\omega|^2.$ Due to Proposition \ref{LichnerowiczAsTrace}, we may hence assume $\kappa = 0.$

For $\kappa =0,$ Proposition \ref{LichnerowiczAsTrace} shows that $\tr(\mathcal{R} \mathcal{A}_{\omega}) \geq 0$ for all $k$-nonnegative algebraic curvature operators $\mathcal{R}$. Thus,
\begin{align*}
    \mathcal{A}_{\omega} \in \left( S_k \cap \Sym_B^2( \Lambda^2V) \right)^* = S_k^* + \widehat{\Lambda^4V^*}
\end{align*}
due to Proposition \ref{DualCones} (b).

In particular, there are $A \in S_k^*$ and $\eta \in \Lambda^4V^*$ such that $A = \mathcal{A}_{\omega}+\hat{\eta}.$ Proposition \ref{DualCones} (a) applied to $A$ shows that 
\begin{align*}
    0 \leq \mathcal{A}_{\omega}+\hat{\eta} \leq \frac{\tr(\mathcal{A}_{\omega}+\hat{\eta})}{k}\id_{\Lambda^2V}= \frac{\tr(\mathcal{A}_{\omega})}{k}\id_{\Lambda^2V} = \frac{p(n-p)}{k}|\omega|^2\id_{\Lambda^2V}
\end{align*}
and hence the claim.
\end{proof}

\begin{remark} \label{Positivity}
\normalfont
    Under the assumption of Theorem \ref{SharpnessCriterion} (b), note that for a $k$-positive algebraic curvature operator the induced Lichnerowicz operator $\Ric_L$ is positive on $\Lambda^p V^{*}.$
\end{remark}

\begin{remark}\label{RemarkOnPWestimates}
    \normalfont 
    If $1 \leq p \leq n/2,$ then the estimate $0 \leq \mathcal{A}_{\omega}   \leq p  | \omega |^2 \id_{\Lambda^2 V}$ in \cite[Lemma 2.2]{PetersenWinkNewCurvatureConditionsBochner} shows that we may take $\eta =0$ to prove that $\Ric_L \geq 0$ on $\Lambda^p V^{*}$ provided $R$ is an $(n-p)$-nonnegative curvature operator.
\end{remark}

\begin{remark}
\normalfont
    If $\omega=e^1$ is a $1$-form, then $\mathcal{A}_{\omega}(e_1 \wedge e_j)=e_1 \wedge e_j$ for $j >1$ and $\mathcal{A}_{\omega} (e_i \wedge e_j)=0$ for $i,j>1.$ In particular, $0 \leq \mathcal{A}_{\omega} \leq |\omega|^2 \id_{\Lambda^2V}$. Moreover, $0 \leq  \mathcal{A}_{\omega} + \hat{\eta} \leq |\omega|^2 \id_{\Lambda^2V}$ implies $\hat{\eta}=0.$ In particular, the result in Remark \ref{RemarkOnPWestimates} is sharp for $p=1.$ In this case $\Ric_L$ is of course the Ricci endomorphism. Example \ref{ExampleSharpCandidate} shows that the estimate in Remark \ref{RemarkOnPWestimates} is also sharp for $p=2.$
\end{remark}

\section{The maximizer of the action estimate}

In order to improve upon the results in \cite{PetersenWinkNewCurvatureConditionsBochner}, see Remark \ref{RemarkOnPWestimates}, we note that the $O(n)$-orbit of the examples given in \cite[Example 4.2]{PetersenWinkNewCurvatureConditionsBochner} is exactly the maximizer of the action estimate $|L\omega|^2 \leq \min \{ p, n-p \} |L|^2 |\omega|^2$ in \cite[Lemma 2.2]{PetersenWinkNewCurvatureConditionsBochner}.

\begin{proposition}
\label{CharacterizationMaximizer}
    Let $1 \leq p \leq \frac{n}{2},$ $0 \neq \omega \in \Lambda^p V^*$ and $0 \neq L \in \mathfrak{so}(V).$ Then 
    \begin{align*}
        | L \omega |^2 = p | L|^2 |\omega|^2
    \end{align*}
    if and only if there exists an orthonormal basis $e_1, \ldots, e_n$ for $V$ such that
    \begin{align*}
        L = \pm \frac{|L|}{\sqrt{p}} \sum_{j=1}^p e_{2j-1} \wedge e_{2j} \ \text{ and } \ \omega = \sqrt{2} |\omega| \operatorname{Re} \varphi^1 \wedge \ldots \wedge \varphi^p
    \end{align*}
    where $\varphi^j= \frac{1}{\sqrt{2}} \left( e^{2j-1}+\sqrt{-1}e^{2j} \right).$

    In this frame, if $L= \sum_{j=1}^p e_{2j-1} \wedge e_{2j},$ then
    \begin{align}
    \label{LActionOnMaximizer}
        L \operatorname{Re} \varphi^1 \wedge \ldots \wedge \varphi^p & = + p \operatorname{Im} \varphi^1 \wedge \ldots \wedge \varphi^p, \\
        L \operatorname{Im} \varphi^1 \wedge \ldots \wedge \varphi^p & = - p \operatorname{Re} \varphi^1 \wedge \ldots \wedge \varphi^p. \nonumber
    \end{align}
\end{proposition}
\begin{proof}
    Suppose that $n=\dim V$ is even. Since $L$ is skew-symmetric, there is an orthonormal basis $e_1, \ldots, e_n$ for $V$ and $a_1, \ldots, a_{\frac{n}{2}} \geq 0$ such that 
    \begin{align*}
        L = \sum_{j=1}^{\frac{n}{2}} a_j e_{2j-1} \wedge e_{2j}. 
    \end{align*}
    We may assume that $|L|^2 = \sum_{j} a_j^2 =1.$ 
     
    Define $\varphi^j = \frac{1}{\sqrt{2}} \left( e^{2j-1} + \sqrt{-1} e^{2j} \right)$ and $\varphi^{IJ} = \varphi^{i_1} \wedge \ldots \wedge \varphi^{i_k} \wedge \overline{\varphi^{j_1}} \wedge \ldots \wedge \overline{\varphi^{j_{p-k}}}$ where $I=\{ i_1 < \ldots < i_k \},$ $J=\{ j_1 < \ldots < j_{p-k} \}.$ As $n$ is even, the collection of all $\varphi^{IJ}$ forms a basis of $\Lambda^p ( V^* \otimes \C)$.
    
    Since $L \varphi^j = - \sqrt{-1} a_j \varphi^j$ and $L \overline{\varphi^j} = + \sqrt{-1} a_j \overline{\varphi^j},$ it follows that 
    \begin{align*}
        L \varphi^{IJ} = \sqrt{-1} \sum_{k=1}^{\frac{n}{2}} \varepsilon_{IJ}^k a_k \varphi^{IJ}
    \end{align*}
    where 
    \begin{align*}
        \varepsilon_{IJ}^k = \begin{cases}
            -1 & \ \text{ if } \ k \in I \setminus J, \\
            +1 & \ \text{ if } \ k \in J \setminus I, \\
            \hspace{2mm} 0 & \ \text{ if } \ k \in I \cap J \text{ or } k \notin I \cup J. \\
        \end{cases}
    \end{align*}
    
    Define $\mu_{IJ} = \sum_k \varepsilon_{IJ}^k a_k$ so that $L \varphi^{IJ} = \sqrt{-1} \mu_{IJ} \varphi^{IJ}.$ Then \cite[Lemma 2.2]{PetersenWinkNewCurvatureConditionsBochner} can be rephrased as 
    \begin{align}
    \label{EqualityCase1}
        |\mu_{IJ}|^2 = | L \varphi^{IJ} |^2 = \left| \sum_{k=1}^{\frac{n}{2}} \varepsilon_{IJ}^{k}  a_k \right|^2
        \leq \sum_{k=1}^{\frac{n}{2}} \left( \varepsilon_{IJ}^k \right)^2 \sum_{\varepsilon_{IJ}^k \neq 0} a_k^2 \leq p \sum_{k=1}^{\frac{n}{2}} a_k^2 = p.
    \end{align}
    Hence, $\omega = \sum_{I,J} \omega_{IJ} \varphi^{IJ} \in \Lambda^p ( V^* \otimes \C)$ satisfies
    $L\omega = \sqrt{-1} \sum_{I,J} \omega_{IJ} \mu_{IJ} \varphi^{IJ}$ and 
    \begin{align*}
        |L\omega|^2 = \sum_{I,J} | \omega_{IJ}|^2 |\mu_{IJ}|^2 \leq p |L|^2 \sum_{I,J} | \omega_{IJ}|^2 = p |\omega|^2.
    \end{align*}
    Equality forces $0 \leq \sum_{I,J} | \omega_{IJ}|^2 \left(p - | \mu_{IJ} |^2  \right) = 0$ and hence either $\omega_{IJ}=0$ or $p = | \mu_{IJ} |^2.$ Therefore, the equality case in \eqref{EqualityCase1} shows that 
    \begin{enumerate}
        \item If $p = | \mu_{IJ}|^2,$ then $p=\sum_{k} \left( \varepsilon_{IJ}^k \right)^2.$ Consequently, exactly $p$ many $\varepsilon_{IJ}^k$ are nonzero.
        \item If $\varepsilon_{IJ}^k=0,$ then $a_k=0.$
        \item We have $(\varepsilon_{IJ}^1, \ldots, \varepsilon_{IJ}^{\frac{n}{2}}) = c \cdot (a_1, \ldots, a_{\frac{n}{2}}),$ according to the equality discussion in Cauchy-Schwarz. Since $\omega \neq 0$ we have $c \neq 0$ and as $a_k \geq 0,$ all nonzero $\varepsilon_{IJ}^k$ have the same sign.
    \end{enumerate}
    Therefore, we may assume that $a_j \neq 0$ if and only if $1 \leq j \leq p$ and then $|L|^2=1$ shows $a_1=\ldots=a_p= \frac{1}{\sqrt{p}}$ and $a_{p+1}= \ldots = a_{\frac{n}{2}}=0.$ This yields the normal form for $L.$

    Moreover, this implies that $\omega = z  \varphi^1 \wedge \ldots \wedge \varphi^p + w \overline{\varphi^1} \wedge \ldots \wedge \overline{\varphi^p}$ and as $\omega$ is a real form, we have $w=\bar{z}.$ Define $t \in \R$ by $\sqrt{2}z=e^{ipt} |\omega|$ and the rotated orthonormal frame $f_1, \ldots, f_n$ by
    \begin{align*}
        f_{2j-1} & = \cos (t) e_{2j-1} - \sin (t) e_{2j} \ \text{ for } \ j=1, \ldots, p, \\
        f_{2j} & = \sin (t) e_{2j-1} + \cos (t) e_{2j} \ \text{ for } \ j=1, \ldots, p , \\
        f_{j} & = e_j \hspace{34.6mm} \text{ for } \ j=2p+1, \ldots, n.
    \end{align*}
    It follows that $L$ and $\omega$ are in the claimed normal form as $f_{2j-1} \wedge f_{2j} = e_{2j-1} \wedge e_{2j}$ and $\frac{1}{\sqrt{2}} \left( f^{2j-1} + \sqrt{-1} f^{2j} \right) = e^{it} \varphi^j.$

    If $n=\dim V$ is odd, consider $V \oplus \R$ with the product metric, extend $L$ by zero and pull back $\omega$ via the orthogonal projection. This does not change the norms of $L,$ $\omega$ and $L\omega.$ We may now apply the proposition for even dimensions. As the added $\R$ factor is in the kernel of $L$, it follows that the basis representation for $L$ consists entirely of elements of $V.$ Equivalently, 
    there is an orthonormal basis $e_1, \ldots, e_n$ for $V$ such that 
    \begin{align*}
        L = \sum_{j=1}^{\ceil{\frac{n}{2}}} a_j e_{2j-1} \wedge e_{2j}
    \end{align*}
    where $a_{\ceil{\frac{n}{2}}}=0$ and $e_{2 \cdot \ceil{\frac{n}{2}}} \in \R$. Now the same proof applies, with $\frac{n}{2}$ replaced by $\ceil{\frac{n}{2}}$ and $V$ by $V \oplus \R$.
\end{proof}

The next step is to obtain a description for the action operator $\mathcal{A}_{\omega}$ for maximizers $\omega \in \Lambda^pV^*$ of $|L\omega|^2 \leq p |L|^2 |\omega|^2$.

\begin{proposition}
\label{ActionOnMaximizer}
    Let $3 \leq p \leq \frac{n}{2}$ and suppose that $\omega \in \Lambda^pV^* \setminus \{ 0 \}$ satisfies $|L \omega|^2=p |L|^2 |\omega|^2$ for some $0 \neq L \in \mathfrak{so}(V).$ Then there exists an orthogonal decomposition $V=E \oplus F$ with $\dim E =2p$ and an almost complex structure $J$ on $E$ such that 
    \begin{align*}
        \Lambda^2V = \R \Omega \oplus \Lambda_{0,\R}^{1,1}E \oplus \left( \Lambda^{2,0}E \oplus \Lambda^{0,2}E\right)_{\R} \oplus E \wedge F \oplus \Lambda^2F,
    \end{align*}
    where $\Omega$ is the bi-vector associated to the K\"ahler form on $E,$ is an orthogonal decomposition into eigenspaces of $\mathcal{A}_{\omega}.$ Moreover, 
    \begin{align*}
        \mathcal{A}_{\omega}(\Omega) & = p |\omega|^2 \Omega, \\
        \mathcal{A}_{\omega} & = 0  \hspace{12.2mm} \text{ on } \ \Lambda^{1,1}_{0, \R} E, \\
        \mathcal{A}_{\omega} & = |\omega|^2 \id \hspace{3.2mm} \text{ on } \ \left( \Lambda^{2,0}E \oplus \Lambda^{0,2}E \right)_{\R}, \\
        \mathcal{A}_{\omega} & = \frac{1}{2} |\omega|^2 \id \hspace{0.6mm} \text{ on } \ E \wedge F, \\
        \mathcal{A}_{\omega} & = 0  \hspace{12.3mm} \text{ on } \ \Lambda^2 F.
    \end{align*}
\end{proposition}
\begin{proof}
    We may assume $|\omega|=1.$ By Proposition \ref{CharacterizationMaximizer}, there exists an orthonormal basis $e_1, \ldots, e_{n}$ for $V$ such that 
    \begin{align*}
        \omega = \sqrt{2} \operatorname{Re}   \varphi^1 \wedge \ldots \wedge \varphi^p,
    \end{align*}
    where $\varphi^j= \frac{1}{\sqrt{2}} \left( e^{2j-1}+\sqrt{-1}e^{2j} \right).$
    
    Set $E=\operatorname{span} \{ e_1, \ldots, e_{2p} \}$ and $F = \operatorname{span} \{ e_{2p+1}, \ldots, e_{n} \}.$ Then $E$ has the almost complex structure $J$ given by $J(e_{2j-1})=e_{2j}$, $J(e_{2j})=-e_{2j-1}$ and $\Omega = \frac{1}{\sqrt{p}} \sum_{j=1}^p e_{2j-1} \wedge e_{2j}$ is the bi-vector associated to the K\"ahler form.   

    (a) Note that \eqref{LActionOnMaximizer} implies $g( \mathcal{A}_{\omega}(\Omega), \Omega) = p,$ which is the equality case in $\mathcal{A}_{\omega} \leq p \id_{\Lambda^2V}$ and hence $\Omega$ is an eigenvector of $\mathcal{A}_{\omega}.$ 

    (b) As $\Lambda_{0,\R}^{1,1}E \cong \mathfrak{su}(E)$ and $\varphi^1 \wedge \ldots \wedge \varphi^p$ is the complex volume form on $E,$ which is fixed by $SU(E),$ it follows that $(L)\varphi^1 \wedge \ldots \wedge \varphi^p = 0$ and hence $L \omega=0$ for all $L \in \Lambda_{0,\R}^{1,1}E.$ 

    (c) As in \cite[Section 1.2]{PetersenWinkHodgeNumbers}, an orthonormal basis for $\left( \Lambda^{2,0}E \oplus \Lambda^{0,2}E \right)_{\R} = \mathfrak{u}(E)^\perp \subseteq \mathfrak{so}(E)$ is given by
    \begin{align}
    \label{BasisUnPerp}
        (R_{ij})^{\perp} & = \frac{1}{\sqrt{2}} \left( e_{2i-1} \wedge e_{2j-1} - e_{2i} \wedge e_{2j} \right), \\ 
        (I_{ij})^{\perp} & = \frac{1}{\sqrt{2}} \left( e_{2i-1} \wedge e_{2j} - e_{2j-1} \wedge e_{2i} \right). \notag
    \end{align}
    It is straightforward to compute that
    \begin{align*}
        (R_{ij})^{\perp} \varphi^k & =   \frac{1}{\sqrt{2}} \left( \delta_{ik} \overline{\varphi^j} - \delta_{jk} \overline{\varphi^i} \right), \ \
        (I_{ij})^{\perp} \varphi^k = \frac{\sqrt{-1}}{\sqrt{2}} \left( \delta_{ik} \overline{\varphi^j} - \delta_{jk} \overline{\varphi^i} \right)
    \end{align*}
    and in particular $(I_{ij})^{\perp} \varphi^k = \sqrt{-1} (R_{ij})^{\perp} \varphi^k.$ It follows that 
    \begin{align*}
        (R_{ij})^{\perp} \varphi^1 \wedge \ldots \wedge \varphi^p & = \frac{1}{\sqrt{2}} \left( \varphi^1 \wedge \ldots \wedge \underset{\text{$i$-th slot}}{\overline{\varphi^j}} \wedge \ldots \wedge \varphi^p \right. \\
        &\hspace{20mm} \left. - \varphi^1 \wedge \ldots \wedge \underset{\text{$j$-th slot}}{\overline{\varphi^i}} \wedge  \ldots \wedge \varphi^p \right).
    \end{align*}
    Set $\varphi = \varphi^1 \wedge \ldots \wedge \varphi^p.$ Note that $(R_{ij})^{\perp} \varphi \in \Lambda^{p-1,1}(E \otimes \C)^*$ and $(R_{ij})^{\perp} \overline{\varphi} \in \Lambda^{1,p-1}(E \otimes \C)^*$ are mutually orthogonal with unit norm. As $p \geq 3,$ the forms are in different spaces and with $\omega = \frac{1}{\sqrt{2}} \left( \varphi + \overline{\varphi} \right)$ it follows that 
    \begin{align*}
         |(R_{ij})^{\perp} \omega|^2 = \frac{1}{2}\left( |(R_{ij})^{\perp}\varphi|^2 + |(R_{ij})^{\perp}\overline{\varphi} |^2 \right)  =1.
    \end{align*}     
    With the same argument, one observes that $|(I_{ij})^{\perp} \omega|^2 =1$ and 
    \begin{align*}
        g( (R_{ij})^{\perp} \omega, (R_{kl})^{\perp} \omega) = g( (I_{ij})^{\perp} \omega, (I_{kl})^{\perp} \omega) = g( (R_{ij})^{\perp} \omega, (I_{kl})^{\perp} \omega)= 0
    \end{align*}
    for $\{ i,j \} \neq \{k,l \}.$ Together with
    \begin{align*}
        2 g( (R_{ij})^{\perp} \omega, (I_{ij})^{\perp} \omega) = \sqrt{-1} g( (R_{ij})^{\perp} \varphi + (R_{ij})^{\perp} \overline{\varphi}, (R_{ij})^{\perp} \varphi - (R_{ij})^{\perp} \overline{\varphi})=0
    \end{align*}
    this shows that $\mathcal{A}_{\omega}= \id$ on $\left( \Lambda^{2,0}E \oplus \Lambda^{0,2}E \right)_{\R}.$

    (d) For $x,z \in E$ and $y,w \in F$ we have $(x \wedge y) \omega = y^{\flat} \wedge \iota_{x}\omega$ and thus 
    \begin{align*}
        g( \mathcal{A}_{\omega}(x \wedge y), z \wedge w) = g(y,w) g(\iota_x \omega, \iota_z \omega) = \frac{1}{2} g(x,z) g(y,w) 
    \end{align*}
    by direct evaluation on basis vectors. Hence, $(\mathcal{A}_{\omega})_{|E \wedge F}= \frac{1}{2} \id_{E \wedge F}.$

    (e) $\mathcal{A}_{\omega}=0$ on $\Lambda^2F$ is immediate as $(x \wedge y) \omega=0$ for all $x,y \in F.$
\end{proof}

\begin{lemma}
\label{Correction4Form}
    Let $p \leq \frac{n}{2}$ and let $e_1, \ldots, e_n$ be an orthonormal basis for $V.$ Let $E=\operatorname{span} \{ e_1, \ldots, e_{2p} \},$ $F = \operatorname{span} \{ e_{2p+1}, \ldots, e_{n} \}.$ 
    
    If $\Omega = \frac{1}{\sqrt{p}} \sum_{j=1}^p e_{2j-1} \wedge e_{2j},$ then $\eta = \frac{p}{2} \Omega^{\flat} \wedge \Omega^{\flat}$ satisfies
    \begin{align*}
        \hat{\eta}(\Omega) & = (p-1) \Omega, \\
        \hat{\eta} & = - \id \hspace{11mm} \text{ on } \ \Lambda^{1,1}_{0, \R} E, \\
        \hat{\eta} & = \id \hspace{14.6mm}  \text{ on } \ \left( \Lambda^{2,0}E \oplus \Lambda^{0,2}E \right)_{\R}, \\
        \hat{\eta} & = 0  \hspace{16.6mm} \text{ on } \ E \wedge F, \\
        \hat{\eta} & = 0  \hspace{16.6mm} \text{ on } \ \Lambda^2 F.
    \end{align*}
\end{lemma}
\begin{proof}
    Note that $\eta = \sum_{i < j} e^{2i-1} \wedge e^{2i} \wedge e^{2j-1} \wedge e^{2j}$ and recall that $\hat\eta$ is defined by $g(\hat{\eta}(x \wedge y), z \wedge w) = \eta( x,y,z,w).$ Therefore, direct evaluation yields
    \begin{align*}
        \hat{\eta}( e_{2k-1} \wedge e_{2k} ) = \sum_{l: l \neq k} e_{2l-1} \wedge e_{2l} = \sqrt{p} \Omega - e_{2k-1} \wedge e_{2k}.
    \end{align*}
    In particular, $\hat{\eta}(\Omega) = (p-1) \Omega.$ 
    
    As $\Lambda^{1,1}_{0,\R}E$ consists of all elements of $\Lambda^{1,1}_{\R}E$ orthogonal to $\Omega,$ any $\alpha \in \Lambda^{1,1}_{0,\R}E$ is given by $\alpha = \sum_k a_k e_{2k-1} \wedge e_{2k}$ with $\sum_k a_k=0$ (after possibly changing the frame with a unitary transformation which automatically fixes $\Omega$) and hence
    \begin{align*}
        \hat{\eta}(\alpha) = \sum_k a_k \left( \sqrt{p} \Omega - e_{2k-1} \wedge e_{2k} \right) = - \alpha.
    \end{align*}

    Evaluation on the orthonormal basis for $\left( \Lambda^{2,0}E \oplus \Lambda^{0,2}E \right)_{\R}$ defined in \eqref{BasisUnPerp} proves $\hat{\eta} = \id$ on $\left( \Lambda^{2,0}E \oplus \Lambda^{0,2}E \right)_{\R}.$

    As $\iota_v \Omega^{\flat} = 0$ for every $v \in F,$ it follows that $\hat{\eta} = 0$ on $E \wedge F \oplus \Lambda^2F.$
\end{proof}

\begin{corollary}
\label{OptimalCorrectionForMaximizers}
    Let $3 \leq p \leq \frac{n}{2}.$ If $\omega \in \Lambda^pV^*$ satisfies $|L \omega|^2=p |L|^2 |\omega|^2$ for some $0 \neq L \in \mathfrak{so}(V),$ then there exists $\eta \in \Lambda^4V^*$ such that
    \begin{align*}
        0 \leq \mathcal{A}_{\omega} + \hat{\eta} \leq | \omega |^2 \id_{\Lambda^2V}.
    \end{align*}
\end{corollary}
\begin{proof}
    Define $\Omega$ as in Proposition \ref{ActionOnMaximizer} with $|\Omega|=1$. According to Lemma \ref{Correction4Form}, $\eta = -\frac{p}{2} |\omega|^2 \Omega^{\flat} \wedge \Omega^{\flat}$ works. 
\end{proof}

\begin{remark}
    \normalfont
    For $p=1,2,$ if $\omega \in \Lambda^pV^*$ satisfies $|L\omega|^2=p|L|^2 |\omega|^2$ for some $0 \neq L \in \mathfrak{so}(V),$ then $\mathcal{A}_{\omega} + \hat{\eta} \geq 0$ for some $\eta \in \Lambda^4V^*$ implies $\eta =0.$
\end{remark}

\begin{theorem}
\label{EpsilonImprovement}
    For every $3 \leq p \leq \frac{n}{2}$ there exists $\varepsilon(p,n)>0$ with the following property. If $\omega \in \Lambda^pV^*,$ then there exists $\eta \in \Lambda^4V^*$ such that 
    \begin{align*}
        0 \leq \mathcal{A}_{\omega} + \hat{\eta} \leq (p - \varepsilon) | \omega|^2 \id_{\Lambda^2V}.
    \end{align*}
\end{theorem}
\begin{proof}
    (a) For a maximizer $\omega_{\max} \in \Lambda^pV^*$ of the estimate $|L \omega |^2 \leq p |L|^2 |\omega|^2,$ Proposition \ref{ActionOnMaximizer} yields an orthonormal basis $e_1, \ldots, e_n$ for $V$ and a decomposition $V=E \oplus F$ with $E=\operatorname{span} \{ e_1, \ldots, e_{2p} \},$ $F=\operatorname{span} \{ e_{2p+1}, \ldots, e_n \}$ which induces a decomposition of $\Lambda^2V$ into eigenspaces of $\mathcal{A}_{\omega_{\max}}.$ We may assume $| \omega_{\max}|=1.$ Then every unit norm maximizer is contained in the $O(V)$ orbit of $\omega_{\max}$
    \begin{align*}
        \mathcal{O} = O(V) \cdot \omega_{\max} & = \{  \omega \in \Lambda^p V^* \ | \ | \omega |=1,  \text{ there is } L \in \mathfrak{so}(V) \text{ with } \\ 
        & \hspace{28mm} L \neq 0 \text{ and } |L \omega |^2 = p |L|^2  \}.
    \end{align*}
    If $\Omega = \frac{1}{\sqrt{p}} \sum_{j=1}^p e_{2j-1} \wedge e_{2j},$ then $T = \mathcal{A}_{\omega_{\max}} - \frac{p}{4} \widehat{ \Omega^{\flat} \wedge \Omega^{\flat}}$ satisfies
    \begin{align*}
        T & = \frac{p+1}{2} \id \hspace{3.8mm} \text{ on } \ \R \Omega, \\
        T & = \frac{1}{2} \id \hspace{10.3mm} \text{ on } \ \Lambda^{1,1}_{0} E \oplus \left( \Lambda^{2,0}E \oplus \Lambda^{0,2}E \right)_{\R} \oplus \ E \wedge F, \\
        T & = 0  \hspace{15.7mm} \text{ on } \ \Lambda^2 F
    \end{align*}
    according to Proposition \ref{ActionOnMaximizer} and Lemma \ref{Correction4Form}.

    (b) The idea is to use the correction operator $\mathcal{A}_{\omega} - \frac{p}{4} |\omega|^2 \widehat{\Omega^{\flat} \wedge \Omega^{\flat}}$ also for forms $\omega$ close to the orbit $\mathcal{O}.$ The key step is to prove that the perturbed operator is nonnegative on $\Lambda^2F$. 
    
    Suppose that $\omega \in \Lambda^p V^*$ has unit norm. Pick $\omega_{\max} \in \mathcal{O}$ minimizing the distance to $\omega$ and $\mu \in \Lambda^p V^*$ such that $\omega = \omega_{\max} + \mu.$ Since $\mu$ is orthogonal to $T_{\omega_{\max}} \mathcal{O}$ and the tangent space is $T_{\omega_{\max}}\mathcal{O} = \mathfrak{so}(V) \cdot \omega_{\max},$ it follows that $g(\alpha \cdot \omega_{\max}, \mu)=0$ for all $\alpha \in \Lambda^2V.$ In view of Proposition \ref{PropertiesBilinearOperator}, to determine nonnegativity of $\mathcal{A}_{\omega_{\max} + \mu}= \mathcal{A}_{\omega_{\max}} + \mathcal{B}_{\omega_{\max}, \mu} + \mathcal{A}_{\mu}$ consider $\mathcal{B}_{\omega_{\max}, \mu}.$ By the previous observation and Proposition \ref{PropertiesBilinearOperator} (d), we have
    \begin{align*}
        g( \mathcal{B}_{\omega_{\max}, \mu}(\alpha), \beta) = 2g( \alpha \cdot \omega_{\max}, \beta \cdot \mu) -  g([\alpha, \beta] \cdot \omega_{\max}, \mu)= 2 g( \alpha \cdot \omega_{\max}, \beta \cdot \mu) 
    \end{align*}
    for all $\alpha, \beta \in \Lambda^2V.$ By construction, $\iota_v\omega_{\max} = 0$ for all $v \in F$ and hence $\alpha \cdot \omega_{\max} =0$ for all $\alpha \in \Lambda^2F.$ This implies $\mathcal{B}_{\omega_{\max},\mu} = 0$ on $\Lambda^2F.$

    Therefore, for all $\omega \in \Lambda^pV^*,$ the correction operator satisfies
    \begin{align*}
        g( (\mathcal{A}_{\omega} - \frac{p}{4} \widehat{\Omega^{\flat} \wedge \Omega^{\flat}})(\alpha), \alpha) =  g(\mathcal{A}_{\mu}(\alpha), \alpha) \geq 0
    \end{align*}
    for all $\alpha \in \Lambda^2F.$

    (c) Pick $r >0$ such that $2pr < \frac{1}{2}$ and $\varepsilon_1 = \frac{p-1}{2}-pr^2-2pr > 0.$ Let $\mathcal{U}$ consist of all unit norm $\omega \in \Lambda^pV^*$ of distance at most $r>0$ from $\mathcal{O}.$ As in (b), $\omega \in \mathcal{U}$ determines $\omega_{\max} \in \mathcal{O},$ $\mu = \omega - \omega_{\max}$ with $|\mu|<r,$ the decomposition $V = E \oplus F$ and $\Omega \in \Lambda^2E.$ For $v \in \Lambda^2V$ write $v=v_0 + v_{\Lambda^2F}$ with $v_0$ orthogonal to $\Lambda^2F.$ Then (a) and Proposition \ref{PropertiesBilinearOperator} show 
    \begin{align*}
        g((\mathcal{A}_{\omega} - \frac{p}{4} \widehat{\Omega^{\flat} \wedge \Omega^{\flat}})(v),v) =  & \ g((\mathcal{A}_{\omega_{\max}} - \frac{p}{4} \widehat{\Omega^{\flat} \wedge \Omega^{\flat}})(v_0),v_0) \\
        & \ + g(\mathcal{B}_{\omega_{\max},\mu}(v_0),v_0) + g(\mathcal{A}_{\mu}(v),v) \\
        \geq & \ \left( \frac{1}{2}-2p|\mu| \right) |v_0|^2 \geq 0
    \end{align*}
    and similarly 
    \begin{align*}
        g((\mathcal{A}_{\omega} - \frac{p}{4} \widehat{\Omega^{\flat} \wedge \Omega^{\flat}})(v),v) \leq \frac{p+1}{2} | v_0|^2 + 2p|\mu| |v_0|^2 + p |\mu|^2 |v|^2  \leq \left( p - \varepsilon_1 \right) |v|^2.
    \end{align*}
    This implies that for all $\omega \in \mathcal{U}$ there exists $\eta \in \Lambda^4V^*$ such that
    \begin{align*}
        0 \leq \mathcal{A}_{\omega} + \hat{\eta} \leq (p-\varepsilon_1)  \id_{\Lambda^2V}.
    \end{align*}

    (d) As $K = \{ \omega \in \Lambda^pV^* \ | \ | \omega|=1 \} \setminus \mathcal{U}$ is compact, there is $\varepsilon_2 > 0$ such that
    \begin{align*}
        0 \leq \mathcal{A}_{\omega} \leq (p-\varepsilon_2) \id_{\Lambda^2V}
    \end{align*}
    for all $\omega \in K.$ Taking $\varepsilon=\min \{ \varepsilon_1, \varepsilon_2 \} >0$ yields the claim. 
\end{proof}

\textit{Proof of Theorem  \ref{GeneralRigidityTheorem}.} Let $3 \leq p \leq \frac{n}{2}$ and pick $\varepsilon_0(p,n)>0$ according to Theorem \ref{EpsilonImprovement}. Let $\omega \in \Lambda^pT^*M$ be harmonic. Then Theorem \ref{EpsilonImprovement} shows that there exists $\eta \in \Lambda^4T^*M$ such that $0 \leq \mathcal{A}_{\omega} + \hat{\eta} \leq (p-\varepsilon_0) | \omega |^2 \id_{\Lambda^2TM}.$ Thus, Theorem \ref{SharpnessCriterion} applies with $k=\frac{p(n-p)}{p-\varepsilon_0(p,n)}=n-p+\varepsilon(p,n).$ In particular, if the curvature operator $\mathcal{R}$ is $\left( n-p+\varepsilon \right)$-nonnegative (respectively $\left( n-p+\varepsilon \right)$-positive), then $\Ric_L \geq 0$ (respectively $\Ric_L > 0)$ due to Proposition \ref{LichnerowiczAsTrace} and Remark \ref{Positivity}. As every harmonic form satisfies 
\begin{align*}
    \Delta \frac{1}{2} | \omega|^2 = | \nabla \omega|^2 + g(\Ric_L(\omega),\omega),
\end{align*}
the maximum principle implies that $\omega$ is parallel (respectively $\omega$ vanishes).

For the estimation theorem, recall that Ricci curvature is bounded from below by the lowest $(n-1)$ eigenvalues of $\mathcal{R}.$ In particular, if the eigenvalues of the curvature operator satisfy $\lambda_1 + \ldots + \lambda_{n-p} + \varepsilon \lambda_{n-p+1} \geq (n-p+\varepsilon) \kappa$, then Ricci curvature is bounded from below and Theorem \ref{SharpnessCriterion} shows that $g(\Ric_L(\omega), \omega) \geq p(n-p)\kappa |\omega|^2.$ The work of P. Li \cite{LiSobolevConstant} and Gallot \cite{GallotSobolevEstimates} now yields the estimation theorem, see also \cite[Theorem 1.9]{PetersenWinkNewCurvatureConditionsBochner}. $\hfill \Box$\vspace{2mm}

Theorem \ref{EpsilonImprovementManifolds} is a special case of Theorem \ref{GeneralRigidityTheorem}. Theorem \ref{GeneralHomologySphereTheorem} follows from Theorem \ref{EpsilonImprovementManifolds} and Poincar\'e duality. 

\begin{remark}
\label{EffectiveEstimateRemark}
    \normalfont
     It is possible to give a quantitative estimate in Theorem \ref{EpsilonImprovement}. In particular, one may pick $\varepsilon_p=\frac{1}{128p^2}$ independent of $n.$ The key step is Proposition \ref{EffectiveEstimate} below, which picks an explicit approximation of $\omega$ in $\mathcal{O}$. This yields the dichotomy that for $\omega \in \Lambda^pV^*$ with unit norm, either $0 \leq \mathcal{A}_{\omega} \leq (p - \varepsilon_p) \id_{\Lambda^2V}$ or Proposition \ref{EffectiveEstimate} applies with $r=\varepsilon_p$ and shows $\dist(\omega,\mathcal{O}) \leq \frac{1}{8p}$. In this case the proof of Theorem \ref{EpsilonImprovement} shows that $0 \leq \mathcal{A}_{\omega} + \hat{\eta} \leq \left( p-\frac{1}{2} \right) \id_{\Lambda^2V} \leq (p-\varepsilon_p) \id_{\Lambda^2V}.$
\end{remark}

\begin{proposition}
\label{EffectiveEstimate}
    Let $1 \leq p \leq \frac{n}{2}$ and  $\omega \in \Lambda^pV^*$ with $|\omega|=1.$ Suppose that there is $L \in \mathfrak{so}(V)$ such that $| L \omega|^2 > (p-r) |L|^2.$ If $\mathcal{O} = \{ \omega \in \Lambda^p V^* \ | \ |\omega|=1, \text{ there is } L \in \mathfrak{so}(V) \text{ with } L\neq 0 \text{ and } |L \omega |^2 = p |L|^2  \},$ then $\dist(\omega, \mathcal{O}) \leq \sqrt{2r}$ for all $0 < r \leq \frac{1}{32p}.$
\end{proposition}
\begin{proof}
    We may assume $p \geq 2$ as every one form is a maximizer of the estimate $|L \omega |^2 \leq p |L|^2 |\omega|^2.$ Indeed, note that for $\omega = e^1$ we may pick $L=e_1 \wedge e_2.$
    
    We may assume that $L$ is a unit length eigenvector corresponding to the maximal eigenvalue of $\mathcal{A}_{\omega}.$ 
    
    Revisit the proof of Proposition \ref{CharacterizationMaximizer}. With the notation from that proof, the condition $| L \omega|^2 > (p-r)$ with $r<1$ implies that there is $\omega_{I_0J_0} \neq 0$ such that $|\mu_{I_0J_0}|^2>(p-r)$ and there are exactly $p$ many nonzero $\varepsilon^k_{I_0J_0}.$ We may assume that $\sum_k \varepsilon_{I_0J_0}^k a_k = \mu_{I_0J_0} >0$ by changing the sign of all $\varepsilon_{I_0J_0}^k$ if needed. 

    Define 
    \begin{align*}
        L_0 & = \frac{1}{\sqrt{p}} \sum_{k} \varepsilon_{I_0J_0}^k e_{2k-1} \wedge e_{2k}, \\
        \varphi^{IJ} & = \varphi^{i_1} \wedge \ldots \wedge \varphi^{i_k} \wedge \overline{\varphi^{j_1}} \wedge \ldots \wedge \overline{\varphi^{j_{p-k}}},
    \end{align*}
    and recall that the $\varphi^{IJ}$ diagonalize $L,$ $L \varphi^{IJ} = \sqrt{-1} \mu_{IJ}(L) \varphi^{IJ}.$ Furthermore, it follows that $|L_0|=1$ and $L_0 \varphi^{I_0J_0} = \frac{\sqrt{-1}}{\sqrt{p}} \sum_k \left( \varepsilon_{I_0J_0}^k \right)^2 \varphi^{I_0J_0} = \sqrt{-1}\sqrt{p} \varphi^{I_0J_0}.$ 
    
    Moreover, the $\varphi^{IJ}$ also diagonalize $L_0,$ $L_0 \varphi^{IJ} = \sqrt{-1} \mu_{IJ}(L_0) \varphi^{IJ},$ with $\mu_{IJ}(L_0)= \frac{1}{\sqrt{p}}\sum_k \varepsilon_{IJ}^k \varepsilon_{I_0 J_0}^k.$ It follows that either $|\mu_{IJ}(L_0)|= \sqrt{p}$ and $\varphi^{IJ} = \varphi^{I_0J_0}, \overline{\varphi^{I_0J_0}}$ or $|\mu_{IJ}(L_0)| \leq \frac{p-1}{\sqrt{p}}.$

    Now note that $| L - L_0|^2 = 2 -2 g(L,L_0)=2\left( 1- \frac{\mu_{I_0J_0}}{\sqrt{p}} \right) \leq \frac{2r}{p}.$ Hence, Cauchy-Schwarz implies 
    \begin{align*}
        |\mu_{IJ}(L)-\mu_{IJ}(L_0)| = \left| \sum_k \varepsilon_{IJ}^k \left( a_k - \frac{\varepsilon_{I_0J_0}^k}{\sqrt{p}} \right) \right| \leq \sqrt{p} | L - L_0 | \leq \sqrt{2 r}.
    \end{align*}
    It follows for $(I,J) \neq (I_0,J_0),$ $ (J_0, I_0)$, which correspond to $\varphi^{I_0J_0}$ and $ \overline{\varphi^{I_0J_0}}$, that 
    \begin{align*}
        | \mu_{IJ}(L) | \leq \frac{p-1}{\sqrt{p}} + \sqrt{2r} \leq \sqrt{p-1}.
    \end{align*}

    Let $P$ be the projection onto the span of $\sqrt{2} \operatorname{Re} \varphi^{I_0J_0}, \sqrt{2} \operatorname{Im} \varphi^{I_0J_0} \in \mathcal{O}.$ Then $\omega = P \omega + (\omega-P\omega)$ and the eigenvalue estimates yield
    \begin{align*}
        p-r < \lambda_{\max}( \mathcal{A}_{\omega}) & = | L\omega|^2 \leq p |P\omega|^2 + (p-1) (1-|P\omega|^2) \\
        & = p-1 + |P\omega|^2
    \end{align*}
    and hence $|P \omega|^2 > 1 - r.$ It follows that
    \begin{align*}
        \left|\omega - \frac{P \omega}{|P \omega|}\right|^2 = 2( 1 - |P \omega|) \leq 2 ( 1 - \sqrt{1-r}) \leq 2r
    \end{align*}
    as claimed.
\end{proof}

\section{Three-forms in dimension six}

Let $(V,g)$ be a $6$-dimensional Euclidean vector space. In this section we exhibit a normal form for $\omega \in \Lambda^3V^*$ with respect to the $O(V)$ action on $\Lambda^3V^*.$ We note that there is a normal form with respect to the $GL(V)$ action due to Reichel \cite{ReichelDissertation} and Bryant \cite{BryantGeometryAlmostComplexSixManifolds}. The approach presented here builds up on the work of Hitchin \cite{HitchinThreeFormsSixDim}.  

Choose an orientation of $V.$ The induced Hodge star operator satisfies $*^2=-\id$ on $\Lambda^3V^*$. We extend it complex linearly to $\Lambda^3(V^{*} \otimes \C).$

\begin{lemma} \label{decomposableForms}
    For every decomposable complex $3$-form $\alpha \in \Lambda^3(V^{*} \otimes \C)$ there exist a positively oriented orthonormal basis $e_1, \ldots, e_6$ for $V,$ $z \in \C$ and $a_1, \ldots, a_6 \in \R$ with $a_i^2 + a_{i+3}^2=1$ and $a_i \geq |a_{i+3} |\geq 0$ for $i=1,2,3$ such that $\alpha = z \alpha_0$ where 
    \begin{align*}
        \alpha_0 =  (a_1e^1 + \sqrt{-1} a_4 e^4) \wedge (a_2e^2 + \sqrt{-1} a_5 e^5) \wedge (a_3e^3 + \sqrt{-1} a_6 e^6).
    \end{align*}
\end{lemma}
\begin{proof}
    As $\alpha$ is decomposable, there are $z \in \C$ and $u_1, u_2, u_3 \in V \otimes \C$ with $g(u_i, \overline{u_j})=\delta_{ij}$ such that $\alpha = z u^1 \wedge u^2 \wedge u^3.$ As $\alpha$ only depends on the subspace spanned by $u_1, u_2, u_3,$ we may diagonalize $(g(u_i, u_j))_{i,j}$ with a unitary transformation and assume that in addition $g(u_i, u_j)= c_i \delta_{ij}$ with $c_i \geq 0,$ using the Takagi normal form of a complex symmetric matrix. Moreover, as $u_1, u_2, u_3$ are unitary, Cauchy-Schwarz implies $0 \leq c_i \leq 1$ for $i=1,2,3.$ 

    Set $u_i = x_i + \sqrt{-1} y_i$ with $x_i, y_i \in V.$ Then $g(u_i, \overline{u_j})=\delta_{ij}$ and $g(u_i, u_j)= c_i \delta_{ij}$ yield 
    \begin{align*}
        g(x_i, x_j) + g(y_i, y_j) & = \delta_{ij} \\
        g(x_i, x_j) - g(y_i, y_j) & = c_i \delta_{ij} \\
        g(x_i,y_j) & = 0.
    \end{align*}
    In particular, $|x_i|^2=(1+c_i)/2,$ $|y_i|^2=(1-c_i)/2$ and $x_1, x_2, x_3, y_1, y_2, y_3$ are mutually orthogonal. Normalization and orientation induce the required orthonormal basis $e_1, \ldots, e_6$ (extending appropriately if $y_i=0$). Note that $a_i^2=|x_i|^2$ and $a_{i+3}^2=|y_i|^2$ for $i=1,2,3$ and the claim follows.
\end{proof}

In the following we use the notation $e^{ijk}=e^i \wedge e^j \wedge e^k.$

\begin{corollary}
\label{ConversionToRealForm}
    For every decomposable complex $3$-form $\alpha \in \Lambda^3(V^{*} \otimes \C)$ there exist a positively oriented orthonormal basis $e_1, \ldots, e_6$ for $V,$ $t \in \R$ and $a,b,c,d \in \R$ with
\begin{align}
\label{CoefficientInequalities}
    a \leq b \leq c \leq d, \ 0 \leq b+c,  \ a+c+d \leq  b
\end{align}
such that 
\begin{align*}
    \alpha - \sqrt{-1} * \alpha = e^{it} ( \omega_0 - \sqrt{-1} * \omega_0 ) 
\end{align*}
where 
\begin{equation*}
    \omega_0 = a e^{123} + be^{156} - c e^{246} + d e^{345}.
\end{equation*}
\end{corollary}
\begin{proof}
With the notation in Lemma \ref{decomposableForms}, set $z = r e^{it}$  with $r \geq 0$ and define 
\begin{align*}
    a & = - r( a_1 a_2 a_3 + a_4 a_5 a_6), \ \
    b  = + r (a_1 a_5 a_6 + a_2 a_3 a_4), \\
    c & = + r( a_2 a_4 a_6 + a_1 a_3 a_5), \ \
    d = + r (a_3 a_4 a_5 + a_1 a_2 a_6) 
\end{align*}
and 
\begin{align*}
    \omega_0 = a e^{123} + b  e^{156} - c e^{246} + d  e^{345}.
\end{align*}
Then the normal form in Lemma \ref{decomposableForms} shows that $\alpha = z \alpha_0$ satisfies
\begin{align*}
\alpha - \sqrt{-1} * \alpha = e^{i(t+\pi)} \left( \omega_0 - \sqrt{-1} * \omega_0 \right).
\end{align*}

It remains to show that the coefficients $a,b,c,d \in \R$ can be assumed to satisfy \eqref{CoefficientInequalities}, after possibly choosing a different orthonormal basis. An elementary computation confirms that 
\begin{align*}
    \nu_4 : = -a + b + c + d & = r (a_1 + a_4)(a_2 + a_5) (a_3 + a_6) \geq 0, \\
    \nu_3 : = -a - b - c + d & = r (a_1 - a_4)(a_2 - a_5) (a_3 + a_6) \geq 0, \\
    \nu_2 : = -a - b + c - d & = r (a_1 - a_4)(a_2 + a_5) (a_3 - a_6) \geq 0, \\
    \nu_1 : = -a + b - c - d & = r (a_1 + a_4)(a_2 - a_5) (a_3 - a_6) \geq 0
\end{align*}
since $a_i \geq | a_{i+3}| \geq 0$ for $i=1,2,3.$ 

Note that changing the sign of any two among $e_4,$ $e_5,$ $e_6$ in the basis $e_1, \ldots, e_6$ gives rise to an orientation preserving coordinate change which changes the sign of the corresponding $a_4,$ $a_5,$ $a_6.$ As a result, we may arrange that $\nu_4 \geq \nu_i$ for $i=1,2,3.$ Similarly, note that  permuting the pairs $(e_1, e_4),$ $(e_2,e_5),$ $(e_3,e_6)$ in the basis $e_1, \ldots, e_6$ gives rise to an orientation preserving coordinate change which exchanges the roles of the corresponding pairs $(a_i, a_{i+3}).$ In particular, this permutes $\nu_1, \nu_2, \nu_3$ and hence we may assume $0 \leq \nu_1 \leq \nu_2 \leq \nu_3 \leq \nu_4.$ 
Then the inequalities on $\nu_i$ imply
\begin{align*}
b+c & = \frac{1}{2} \left( \nu_4 - \nu_3 \right) \geq 0, \
d-c = \frac{1}{2} \left( \nu_3 - \nu_2 \right) \geq 0, \\
c-b & = \frac{1}{2} \left( \nu_2 - \nu_1 \right) \geq 0, \
b-a= \frac{1}{2} \left( \nu_4 + \nu_1 \right) \geq 0
\end{align*}
and with $-a+b-c-d = \nu_1 \geq 0$ this yields the claim.
\end{proof}

\begin{definition} \label{DefRotationAndIsomorphism} \normalfont
    (a) For $t \in \R$ define $R_t \colon \Lambda^3V^{*} \to \Lambda^3V^{*}$ by
\begin{align*}
    R_t \omega = ( \cos t ) \omega + ( \sin t) * \omega.
\end{align*}
    (b) Set $\Lambda_{i} = \{ \alpha \in \Lambda^3 (V^* \otimes \C) \ | *\alpha = \sqrt{-1} \alpha \}$ and let $L \colon \Lambda^3 V^{*} \to \Lambda_{i}$ denote the isomorphism
    \begin{align*}
        L(\omega) = \omega - \sqrt{-1} * \omega.
    \end{align*}
\end{definition}

\begin{proposition}
\label{PropertiesRotation}
The map $R_t$ has the following properties:
\begin{enumerate}
    \item $| R_t \omega | = |\omega|$  
    \item $\mathcal{A}_{R_t\omega} = \mathcal{A}_{\omega}$
    \item $L(R_t \omega) = e^{it} L(\omega)$
\end{enumerate}
\end{proposition}
\begin{proof}
    (a) is a direct computation. For (b) recall that the Hodge star operator commutes with the Lie algebra action on forms, i.e., $*(\Xi \omega) = \Xi (* \omega)$ for $\Xi \in \Lambda^2V \cong \mathfrak{so}(V).$ It follows that $\Xi (R_t \omega) = R_t (\Xi \omega).$ As $R_t$ is an orthogonal map by (a), the defining property of $\mathcal{A}_{\omega}$ in Definition \ref{ActionOperator} implies the claim. (c) is a short computation again.
\end{proof}

\begin{theorem}
\label{NormalForm}
    Let $(V,g)$ be an oriented real $6$-dimensional Euclidean vector space. For every $\omega \in \Lambda^3V^*$ there exist a positively oriented orthonormal basis $e_1, \ldots, e_6$ for $V,$ $t \in \R$ and $a,b,c,d \in \R$ with
    \begin{align*}
    a \leq b \leq c \leq d, \ 0 \leq b+c,  \ a+c+d \leq  b
\end{align*}
such that
    \begin{align*}
        R_t \omega = \omega_0 = a e^{123} + be^{156} - c e^{246} + d e^{345}.
    \end{align*}
\end{theorem}
\begin{proof}
    For a given $\Omega \in \Lambda^3 (V^{*} \otimes \C)$, Hitchin defines $\lambda( \Omega) \in \left(\Lambda^6 (V^{*} \otimes \C)\right)^{\otimes 2}$ in \cite{HitchinThreeFormsSixDim}. 
    
    Suppose that $\lambda( L(\omega)) \neq 0.$ Then, by \cite[Section 2.1, Proposition 1]{HitchinThreeFormsSixDim} there exist decomposable complex three-forms $\alpha, \beta$ with $\alpha \wedge \beta \neq 0$ such that $L(\omega) = \alpha + \beta$ and $\alpha, \beta$ are unique up to reordering. The definition of the isomorphism $L$ implies that $\sqrt{-1} \alpha + \sqrt{-1} \beta = \sqrt{-1} L(\omega) = * L(\omega) = * \alpha + * \beta.$ Since $* \alpha, *\beta$ are decomposable again, the uniqueness part of Hitchin's result implies that either $* \alpha = \sqrt{-1} \alpha$ and $* \beta = \sqrt{-1} \beta,$ or $* \alpha = \sqrt{-1} \beta$ and hence $* \beta = \sqrt{-1} \alpha.$ The former case can be ruled out as it implies $\alpha \wedge \beta = * \alpha \wedge * \beta = - \alpha \wedge \beta,$ but $\alpha \wedge \beta \neq 0.$ It follows that $* \alpha = \sqrt{-1} \beta$ and  $L(\omega) = \alpha - \sqrt{-1} * \alpha.$ With $\omega_0 = a e^{123} + b  e^{156} - c e^{246} + d  e^{345}$, Corollary \ref{ConversionToRealForm} yields
    \begin{align*}
        L(\omega) = \alpha - \sqrt{-1} * \alpha = e^{it} (\omega_0 - \sqrt{-1} * \omega_0 ) = e^{it} L( \omega_0) = L(R_t \omega_0).
    \end{align*}
    This implies $\omega = R_t \omega_0$ and hence $R_{-t}\omega = \omega_0.$ 

    To remove the assumption $\lambda(L(\omega)) \neq 0,$ observe first that \cite[Section 2.1, Proposition 1]{HitchinThreeFormsSixDim} implies conversely that $\Omega_0= L( e^{123}) = e^{123} - \sqrt{-1} e^{456} \in \Lambda_i$ satisfies $\lambda(\Omega_0) \neq 0.$ As $\lambda$ is homogeneous of degree four, $s \mapsto \lambda( L(\omega) + s \Omega_0)$ defines a nonzero polynomial of degree four. This implies that the set $\lambda(\Omega) \neq 0$ is dense in $\Lambda_i.$ Thus, for a given $\omega \in \Lambda^3V^{*}$ we find $\omega_k \to \omega,$ a positively oriented orthonormal basis $e_1, \ldots, e_6$ for $V$ and sequences $O_k \in SO(V),$ $t_k \in \R$ and $a_k, b_k, c_k, d_k \in \R$ such that 
    \begin{align*}
        O_k \cdot R_{t_k} \omega_k = a_k e^{123} + b_k e^{156} - c_k e^{246} + d_k e^{345}.
    \end{align*}
    As $a_k^2+b_k^2+c_k^2+d_k^2=| \omega_k |^2 \to |\omega|^2,$ compactness yields the corresponding normal form for $\omega.$
\end{proof}

For later applications, we collect some basic consequences of the inequalities \eqref{CoefficientInequalities} on the coefficients in the normal form $\omega_0 = a e^{123} + be^{156} - c e^{246} + d e^{345}$. 

\begin{proposition}
\label{InequalitiesForEigenvalues}
    If $a,b,c,d \in \R$ satisfy
\begin{align*}
    a \leq b \leq c \leq d, \ 0 \leq b+c,  \ a+c+d \leq  b,
\end{align*}
then
\begin{align*}
        |b+c| & = b+c \leq d-a = |a-d|,  \ |b-c| = c-b \leq -a-d = |a+d|, \\
        |b+d| & = b+d \leq c-a = |a-c|,  \ |b-d| = d-b \leq -a-c = |a+c|, \\   
        |c+d| & = c+d \leq b-a = |a-b|,  \ |c-d| = d-c \leq -a-b = |a+b|, \\[2mm]
        & \hspace{3mm} -2(bc-ad) \leq -(b+c)^2, \ -2(bd-ac) \leq (b-d)^2.
    \end{align*}
\end{proposition}
\begin{proof}
Note that $0 \leq b+c \leq 2c$ and hence $c \geq 0$ and $|b| \leq c.$ Together with $a+c+d \leq b$ this shows $-d-a \geq c-b \geq 0.$ 

This suffices to establish the computations for the absolute values. Note that $b+c \leq d-a$ since $d-a-b-c=-a-b+d-c \geq -a-d \geq 0.$ Furthermore, $(c-a)-(b+d)=2c-(a+c+d)-b \geq 2(c-b) \geq 0$ and $(d-c) \leq -a-b$ is similar. The other inequalities on the absolute values directly follow from $a+c+d \leq  b.$ Finally, $-(b+c)^2+2(bc-ad) = d^2-b^2+d^2-c^2+2d(-a-d) \geq 0$ and  
$(b-d)^2+2(bd-ac)=b^2 +d^2 -2ac \geq 0$ as $a \leq -d \leq 0$ and $c\geq 0.$
\end{proof}

\section{The Thorpe trick for three-forms in dimension six}

In this section we use the normal form constructed in Theorem \ref{NormalForm} to adjust the action operator $\mathcal{A}_{\omega}$ for every $\omega \in \Lambda^3(\R^6)^*.$ As an application we prove Theorems \ref{SixDimHomologySpheres} and \ref{TheoremThirdBettiNumber}.

\begin{theorem}\label{ThorpeTrick36}
    Let $V$ be a $6$-dimensional Euclidean vector space. For every $\omega \in \Lambda^3 V^{*}$ there exists $\eta \in \Lambda^4V^{*}$ such that
    \begin{align*}
        0 \leq \mathcal{A}_{\omega} + \hat{\eta} \leq 2 |\omega|^2 \id_{\Lambda^2V}.
    \end{align*}
\end{theorem}
\begin{proof}
(a) Let $e_1, \ldots, e_6$ be an orthonormal basis for $V.$ Consider $\omega_0 = ae^{123} + be^{156} - c e^{246} + d e^{345}$ with $a, b,c,d \in \R$ satisfying \eqref{CoefficientInequalities}. To determine $\mathcal{A}_{\omega_0},$ we first compute $(e_i \wedge e_j)\omega_0$ for $1\leq i < j \leq 6.$ We obtain the following block decomposition.
\begin{center}
\begin{tabular}{c|c}
    $e_i \wedge e_j$ & $(e_i \wedge e_j)\omega_0$ \\[1mm]
    \hline & \\[-3mm]
    $12$ & $+c e^{146}+be^{256}$ \\
    $15$ & $-de^{134}+a e^{235}$ \\
    $24$ & $-ae^{134}+de^{235}$ \\
    $45$ & $-be^{146}-ce^{256}$ 
\end{tabular}
\begin{tabular}{c|c}
    $e_i \wedge e_j$ & $(e_i \wedge e_j)\omega_0$   \\[1mm]
    \hline & \\[-3mm]
    $13$ & $-de^{145}+be^{356}$  \\
    $16$ & $+ce^{124}+ae^{236}$ \\
    $34$ & $+ae^{124}+ce^{236}$ \\
    $46$ & $+be^{145}-de^{356}$ 
\end{tabular}
\begin{tabular}{c|c}
    $e_i \wedge e_j$ & $(e_i \wedge e_j)\omega_0$  \\[1mm]
    \hline & \\[-3mm]
    $23$ & $-de^{245}-ce^{346}$ \\
    $26$ & $+be^{125} - ae^{136}$ \\
    $35$ & $+ae^{125}-be^{136}$ \\
    $56$& $+ce^{245}+de^{346}$
\end{tabular}
\end{center}
\begin{center}
\begin{tabular}{c|c}
    $e_i \wedge e_j$ & $(e_i \wedge e_j)\omega_0$  \\[1mm]
    \hline & \\[-3mm]
    $14$ & $-ce^{126}+de^{135}+ae^{234}+be^{456}$ \\
    $25$ & $-be^{126}-ae^{135}-de^{234}+ce^{456}$ \\
    $36$ & $+ae^{126}+be^{135}-ce^{234}+de^{456}$
\end{tabular}
\end{center}
Set $E_0=\operatorname{span}\{ e_1\wedge e_4, e_2 \wedge e_5, e_3\wedge e_6 \}$ and $E_i=\operatorname{span} \{ e_i, e_{i+3} \}$ for $i=1,2,3.$ In particular, $\Lambda^2V = E_0 \bigoplus_{i<j} E_i \wedge E_j$ and $\mathcal{A}_{\omega_0}$ preserves each summand.

With respect to the ordered basis $e_1\wedge e_4,$ $e_2 \wedge e_5,$ $e_3\wedge e_6$ for $E_0,$ the matrix representation of $\mathcal{A}_{\omega_0}$ on $E_0$ is 
\begin{align*}
| \omega_0 |^2 \id + 2
    \begin{pmatrix}
      0 &  bc-ad & bd-ac  \\
      bc-ad & 0 & cd-ab \\
      bd-ac & cd-ab & 0
    \end{pmatrix}.
\end{align*}
On $E_i \wedge E_j$, $\mathcal{A}_{\omega_0}$ easily diagonalizes.
\begin{center}
\begin{tabular}{c|c}
\text{eigenvectors in} $E_1 \wedge E_2$ & \text{eigenvalue}  \\[1mm]
    \hline & \\[-3mm]
    $\frac{1}{\sqrt{2}} \left( e_{1} \wedge e_{2} + e_{4} \wedge e_{5} \right)$ & $(b-c)^2$ \\
    $\frac{1}{\sqrt{2}} \left( e_{1} \wedge e_{2} - e_{4} \wedge e_{5} \right)$ & $(b+c)^2$ \\
    $\frac{1}{\sqrt{2}} \left( e_{1} \wedge e_{5} + e_{2} \wedge e_{4} \right)$ & $(a+d)^2$ \\
    $\frac{1}{\sqrt{2}} \left( e_{1} \wedge e_{5} - e_{2} \wedge e_{4} \right)$ &  $(a-d)^2$
\end{tabular}
\begin{tabular}{c|c}
\text{eigenvectors in} $E_1 \wedge E_3$ & \text{eigenvalue}  \\[1mm]
    \hline & \\[-3mm]
    $\frac{1}{\sqrt{2}} \left( e_{1} \wedge e_{3} + e_{4} \wedge e_{6} \right)$ & $(b-d)^2$ \\
    $\frac{1}{\sqrt{2}} \left( e_{1} \wedge e_{3} - e_{4} \wedge e_{6} \right)$ & $(b+d)^2$ \\
    $\frac{1}{\sqrt{2}} \left( e_{1} \wedge e_{6} + e_{3} \wedge e_{4} \right)$ & $(a+c)^2$ \\
    $\frac{1}{\sqrt{2}} \left( e_{1} \wedge e_{6} - e_{3} \wedge e_{4} \right)$ & $(a-c)^2$
\end{tabular}

\begin{tabular}{c|c}
\text{eigenvectors in} $E_2 \wedge E_3$& \text{eigenvalue}  \\[1mm]
    \hline & \\[-3mm]
    $\frac{1}{\sqrt{2}} \left( e_{2} \wedge e_{3} + e_{5} \wedge e_{6} \right)$ & $(c-d)^2$ \\
    $\frac{1}{\sqrt{2}} \left( e_{2} \wedge e_{3} - e_{5} \wedge e_{6} \right)$ & $(c+d)^2$ \\
    $\frac{1}{\sqrt{2}} \left( e_{2} \wedge e_{6} + e_{3} \wedge e_{5} \right)$ & $(a+b)^2$ \\
    $\frac{1}{\sqrt{2}} \left( e_{2} \wedge e_{6} - e_{3} \wedge e_{5} \right)$ &  $(a-b)^2$
\end{tabular}
\end{center}
(b) We now construct an explicit correction term $\eta \in \Lambda^4V^{*}.$ To this end, define $\eta_{ij} \in \Lambda^4 V^*$ by $\eta_{ij}=e^i \wedge e^{i+3}\wedge e^{j} \wedge e^{j+3}$ and set 
\begin{align*}
    \eta = h_{12} \eta_{12} + h_{13} \eta_{13} + h_{23} \eta_{23} = - h_{12} e^{1245} - h_{13} e^{1346} - h_{23} e^{2356}.
\end{align*}

Then $g(\hat{\eta} (x \wedge y), z \wedge w)=\eta(x,y,z,w)$ implies by explicit evaluation that $\mathcal{A}_{\omega_0} + \hat{\eta}$ preserves the block decomposition from part (a). More precisely, $\mathcal{A}_{\omega_0} + \hat{\eta}$ has the matrix representation
\begin{align*}
    | \omega_0 |^2 \id + 2
    \begin{pmatrix}
      0 &  bc-ad & bd-ac  \\
      bc-ad & 0 & cd-ab \\
      bd-ac & cd-ab & 0
    \end{pmatrix}
    + 
    \begin{pmatrix}
      0 &  h_{12} & h_{13}  \\
      h_{12} & 0 & h_{23} \\
      h_{13} & h_{23} & 0
    \end{pmatrix}
\end{align*}
on $E_0.$ On $E_i \wedge E_j,$ note that $\hat{\eta}=h_{ij} \widehat{\eta_{ij}}$ and one checks that
\begin{align*}
    \frac{1}{\sqrt{2}} \left( e_{i} \wedge e_{j} - e_{i+3} \wedge e_{j+3} \right), \ \frac{1}{\sqrt{2}} \left( e_{i} \wedge e_{j+3} - e_{j} \wedge e_{i+3} \right) 
\end{align*}
are eigenvectors of $\widehat{\eta_{ij}}$ with eigenvalue $+1$ and
\begin{align*}
    \frac{1}{\sqrt{2}} \left( e_{i} \wedge e_{j} + e_{i+3} \wedge e_{j+3} \right), \ \frac{1}{\sqrt{2}} \left( e_{i} \wedge e_{j+3} + e_{j} \wedge e_{i+3} \right) 
\end{align*}
are eigenvectors of $\widehat{\eta_{ij}}$ with eigenvalue $-1.$ In particular, this basis simultaneously diagonalizes $\mathcal{A}_{\omega_0}$ and $\widehat{\eta_{ij}}$ on $E_i \wedge E_j.$ 

\textit{Guiding principle for the choice of $h_{ij}$.} Note that there is an obvious choice for $h_{ij}$ that yields $(\mathcal{A}_{\omega_0} + \hat{\eta})_{|E_0} = | \omega_0 |^2 \id_{E_0}.$ The idea is to choose $h_{ij}$ as close as possible to this ideal choice while still retaining nonnegativity of $\mathcal{A}_{\omega_0} + \hat{\eta}$ on $E_i \wedge E_j.$

With regard to $h_{12},$ observe that Proposition \ref{InequalitiesForEigenvalues} establishes the inequalities $(b+c)^2 \leq (a-d)^2$ and $(b-c)^2 \leq (a+d)^2$ on the eigenvalues of $\mathcal{A}_{\omega_0}$ on $E_1 \wedge E_2.$ As $\mathcal{A}_{\omega_0}$ and $\hat{\eta}$ simultaneously diagonalize on $E_i \wedge E_j$ in the above bases, it is immediate that $\hat{\eta}$ has eigenvalue $+h_{12}$ on the eigenvectors for $(b+c)^2 \leq (a-d)^2$ and eigenvalue $-h_{12}$ on the eigenvectors for $(b-c)^2 \leq (a+d)^2.$ Thus, nonnegativity of $(\mathcal{A}_{\omega_0} + \hat{\eta})$ on $E_1 \wedge E_2$ is equivalent to $-(b+c)^2 \leq h_{12} \leq (b-c)^2.$ Also note that $-2(bc-ad) \leq -(b+c)^2$ holds automatically due to Proposition \ref{InequalitiesForEigenvalues}. The guiding principle therefore explains the definition
\begin{align*}
    h_{12} = -(b+c)^2.
\end{align*}

For $h_{13},$ the same strategy and the inequalities in Proposition \ref{InequalitiesForEigenvalues} yield the condition $-(b+d)^2 \leq h_{13} \leq (b-d)^2$ for nonnegativity of $(\mathcal{A}_{\omega_0} + \hat{\eta})$ on $E_1 \wedge E_3.$ Proposition \ref{InequalitiesForEigenvalues} shows that $-2(bd-ac) \leq (b-d)^2$ and we define 
\begin{align*}
    h_{13} = \begin{cases}
        -(b+d)^2 & \text{ if } -2(bd-ac) \leq -(b+d)^2, \\
         -2(bd-ac) & \text{ if } \hspace{4mm} -(b+d)^2 \leq -2(bd-ac).
    \end{cases}
\end{align*}

For $h_{23},$ we set 
\begin{align*}
    h_{23} = \begin{cases}
        -(c+d)^2 & \text{ if } -2(cd-ab) \leq -(c+d)^2, \\
        -2(cd-ab) & \text{ if } \hspace{4mm} -(c+d)^2 \leq -2(cd-ab) \leq (c-d)^2, \\
        (c-d)^2 & \text{ if } \hspace{8.6mm} (c-d)^2 \leq -2(cd-ab).
    \end{cases}
\end{align*}
With this choice, $-(c+d)^2 \leq h_{23} \leq (c-d)^2$ as required by the guiding principle and the eigenvalue inequalities in Proposition \ref{InequalitiesForEigenvalues}.

(c) The construction in (b) shows that $\mathcal{A}_{\omega_0} + \hat{\eta} \geq 0$ on $E_i \wedge E_j.$ Furthermore, as $\widehat{\eta_{ij}}$ has eigenvalues $\pm 1,$ each with multiplicity two, $\tr_{E_i \wedge E_j}( \mathcal{A}_{\omega_0} + \hat{\eta})=\tr_{E_i \wedge E_j}( \mathcal{A}_{\omega_0})\leq 2|\omega_0|^2$ and hence $0 \leq \mathcal{A}_{\omega_0} + \hat{\eta} \leq 2|\omega_0|^2 \id$ on $E_i \wedge E_j.$

It remains to check that the rows of the matrix 
\begin{align*}
    R= \begin{pmatrix}
    0 & r_{12} & r_{13} \\
    r_{12} & 0 & r_{23} \\
    r_{13} & r_{23} & 0 
    \end{pmatrix}
    =2
    \begin{pmatrix}
      0 &  bc-ad & bd-ac  \\
      bc-ad & 0 & cd-ab \\
      bd-ac & cd-ab & 0
    \end{pmatrix}
    + 
    \begin{pmatrix}
      0 &  h_{12} & h_{13}  \\
      h_{12} & 0 & h_{23} \\
      h_{13} & h_{23} & 0
    \end{pmatrix}
\end{align*}
satisfy
\begin{align} \label{RowInequality}
    |r_{12}|+|r_{13}|, \ |r_{12}|+|r_{23}|, \ |r_{13}|+|r_{23}| \leq |\omega_0|^2.
\end{align}
Indeed, suppose this is the case and let $v=(v_1,v_2,v_3)^T \in \R^3$ with $Rv=\lambda v.$ If $|v_{i}| = \max \{ |v_1|, |v_2|, |v_3| \},$ then $| \lambda | | v_{i} | \leq \sum_{j} |r_{ij} v_j | \leq |\omega_0|^2 |v_i|,$ i.e., $- |\omega_0|^2 \id \leq R \leq |\omega_0|^2 \id$ and thus $0 \leq \mathcal{A}_{\omega_0} + \hat{\eta} \leq 2|\omega_0|^2 \id$ on $E_0$ also.

To prove \eqref{RowInequality}, we run through the different choices for $h_{ij}.$

(c1) For $|r_{12}|+|r_{13}|$ we have two cases according to the choice of $h_{13}$. \\ If $-2(bd-ac) \leq -(b+d)^2,$ then 
\begin{align*}
    |r_{12}|+|r_{13}| & = 2(bc-ad)-(b+c)^2+ 2(bd-ac)-(b+d)^2 \\
    & = |\omega_0|^2 - \frac{1}{4}(a+3b+c+d)^2-(c-d)^2-\frac{3}{4}(-a+b-c-d)^2 \leq | \omega_0|^2.
\end{align*}
If $-(b+d)^2 \leq -2(bd-ac),$ then 
\begin{align*}
    |r_{12}|+|r_{13}| = 2(bc-ad)-(b+c)^2 \leq - 2ad \leq a^2 + d^2 \leq | \omega_0|^2.
\end{align*}

(c2) For  $|r_{12}|+|r_{23}|$ we have three cases. \\
If $-2(cd-ab) \leq -(c+d)^2,$ then 
\begin{align*}
    |r_{12}|+|r_{23}| & = 2(bc-ad)-(b+c)^2 + 2(cd-ab) - (c+d)^2 \\
    & = |\omega_0|^2 - \frac{1}{4} (a+b+3c+d)^2 - (d-b)^2 - \frac{3}{4}(a+b-c+d)^2 \leq |\omega_0|^2.
\end{align*}
If $-(c+d)^2 \leq  -2(cd-ab) \leq (c-d)^2,$ then 
\begin{align*}
    |r_{12}|+|r_{23}| = 2(bc-ad)-(b+c)^2 \leq -2ad \leq a^2 +d^2 \leq |\omega_0|^2.
\end{align*}
If $(c-d)^2 \leq  -2(cd-ab)$, then
\begin{align*}
    |r_{12}|+|r_{23}| & = 2(bc-ad)-(b+c)^2  - 2(cd-ab) - (c-d)^2 \\
    & = |\omega_0|^2 - \frac{1}{4} (a-b-3c+d)^2 - (b+d)^2 - \frac{3}{4}(a-b+c+d)^2 \leq |\omega_0|^2.
\end{align*}
(c3) For  $|r_{13}|+|r_{23}|$ there are six cases. \\
If $-2(bd-ac) \leq -(b+d)^2$ and $-2(cd-ab) \leq -(c+d)^2,$ then
\begin{align*}
    |r_{13}|+|r_{23}| & = 2(bd-ac)-(b+d)^2 + 2(cd-ab) - (c+d)^2 \\
    & = |\omega_0|^2 - \frac{1}{4}(a+b+c+3d)^2 -(c-b)^2-\frac{3}{4}(a+b+c-d)^2 \leq |\omega_0|^2.
\end{align*}
If $-2(bd-ac) \leq -(b+d)^2$ and $-(c+d)^2 \leq -2(cd-ab) \leq (c-d)^2,$ then
\begin{align*}
    |r_{13}|+|r_{23}| = 2(bd-ac)-(b+d)^2 \leq -2ac \leq a^2 + c^2 \leq |\omega_0|^2.
\end{align*}
If $-2(bd-ac) \leq -(b+d)^2$ and $(c-d)^2 \leq -2(cd-ab),$ then
\begin{align*}
    |r_{13}|+|r_{23}| & = 2(bd-ac)-(b+d)^2 - 2(cd-ab) - (c-d)^2 \\
    & = |\omega_0|^2 - \frac{1}{4}(a-b+c-3d)^2 - (b+c)^2- \frac{3}{4}(a-b+c+d)^2 \leq |\omega_0|^2. 
\end{align*}
If $-(b+d)^2 \leq -2(bd-ac)$ and $-2(cd-ab) \leq -(c+d)^2,$
\begin{align*}
    |r_{13}|+|r_{23}| = 2(cd-ab) - (c+d)^2 \leq -2ab \leq a^2 + b^2 \leq | \omega_0|^2.
\end{align*}
If $-(b+d)^2 \leq -2(bd-ac)$ and $-(c+d)^2 \leq -2(cd-ab) \leq (c-d)^2,$ then
\begin{align*}
    |r_{13}|+|r_{23}| = 0.
\end{align*}
If $-(b+d)^2 \leq -2(bd-ac)$ and $(c-d)^2 \leq -2(cd-ab),$ then
\begin{align*}
    |r_{13}|+|r_{23}| = - 2(cd-ab) - (c-d)^2 \leq 2ab \leq a^2+b^2 \leq |\omega_0|^2.
\end{align*}

(d) For $\omega \in \Lambda^3 V^{*}$ pick $O \in SO(V)$ and $t \in \R$ such that $O \cdot R_t \omega = \omega_0$ according to Theorem \ref{NormalForm}. By parts (a) -- (c) there exists $\eta_{\omega_0} \in \Lambda^4V^{*}$ such that 
    \begin{align*}
        0 \leq \mathcal{A}_{\omega_0} + \widehat{\eta_{\omega_0}} \leq 2 | \omega_0 |^2 \id = 2 |\omega|^2 \id.
    \end{align*}
    This implies 
        $0 \leq O^{-1} \cdot \mathcal{A}_{\omega_0}( O \cdot \ ) +  O^{-1} \cdot \widehat{\eta_{\omega_0}}( O \cdot \ ) \leq 2 |\omega|^2 \id$
    and hence \eqref{ActionEquivariance} and Proposition \ref{PropertiesRotation} yield
       \begin{align*}
        0 \leq \mathcal{A}_{\omega} +  \widehat{\eta_{\omega}} \leq 2 |\omega|^2 \id
    \end{align*}
    with $\eta_{\omega}=O^{-1} \cdot \eta_{\omega_0}$.
\end{proof}

\begin{example} \normalfont \label{ExampleSharpCandidate} For $2 \leq p \leq n/2$ there is $\omega \in \Lambda^p V^{*}$ such that 
\begin{align*}
     \mathcal{A}_{\omega} + \hat{\eta} \geq 0 
\end{align*}
for some $\eta \in \Lambda^4V^*$ implies that the maximal eigenvalue of $\mathcal{A}_{\omega} + \hat{\eta}$ is at least $2 |\omega|^2.$ Moreover, for $\eta=0$ the maximal eigenvalue is exactly $2 |\omega|^2.$ In particular, the eigenvalue estimates in Conjecture \ref{ConjectureControllingLichnerowicz} cannot be improved. Specifically, let $e_1, \ldots, e_4$ be a positively oriented basis for the subspace $U \subseteq V$ and set $\alpha = \frac{1}{\sqrt{2}} \left( e^1 \wedge e^2 + e^3 \wedge e^4 \right).$ Consider $\omega = \alpha \wedge e^5 \wedge \ldots \wedge e^{p+2}$. Set $E = \operatorname{span} \{ e_5, \ldots, e_{p+2} \}$ and $F= \operatorname{span} \{ e_{p+3}, \ldots, e_n \}.$ Then we obtain the orthogonal decomposition
\begin{align*}
    \Lambda^2V = \Lambda^2U \oplus \left( U \wedge E \right) \oplus \left( U \wedge F \right) \oplus \Lambda^2E \oplus \left( E \wedge F \right) \oplus \Lambda^2F.
\end{align*}
    Note that $\Lambda^2U$ further decomposes into self-dual and anti-self-dual forms. If
\begin{align*}
    S_1 & = \frac{1}{\sqrt{2}} \left( e^{12} +  e^{34} \right), \ S_2 = \frac{1}{\sqrt{2}} \left( e^{13} - e^{24} \right), \ S_3 = \frac{1}{\sqrt{2}} \left( e^{14} + e^{23} \right), \\
    A_1 & = \frac{1}{\sqrt{2}} \left( e^{12} -  e^{34} \right), \ A_2 = \frac{1}{\sqrt{2}} \left( e^{13} + e^{24} \right), \ A_3 = \frac{1}{\sqrt{2}} \left( e^{14} - e^{23} \right),
\end{align*}
then a computation yields $\mathcal{A}_{\omega} = 2 \id$ on $\operatorname{span} \{ S_2, S_3 \}$ and $\mathcal{A}_{\omega}=0$ on the orthogonal complement in $\Lambda^2U$. Moreover, $\mathcal{A}_{\omega}$ respects the above decomposition of $\Lambda^2V$ and 
\begin{align*}
    (\mathcal{A}_{\omega})_{|U \wedge E} & = \frac{1}{2}\id_{U \wedge E}, \
    (\mathcal{A}_{\omega})_{|U \wedge F}  = \frac{1}{2}\id_{U \wedge F}, \\
    (\mathcal{A}_{\omega})_{|E \wedge F} & = \id_{E \wedge F}, \ (\mathcal{A}_{\omega})_{|\Lambda^2E \oplus \Lambda^2F} = 0.
\end{align*}

Now suppose that $0 \leq \mathcal{A}_{\omega} + \hat{\eta}$ for $\eta \in \Lambda^4V^{*}.$ Consider the restriction to $\Lambda^2U.$ As $U$ is $4$-dimensional, $\eta_{|U} = c \operatorname{vol}_U$ for some constant $c \in \R.$ Testing $ 0 \leq \mathcal{A}_{\omega} + \hat{\eta}$ against $S_1$ and $A_1$ yields $c=0.$ Hence, the largest eigenvalue of $\mathcal{A}_{\omega} + \hat{\eta}$ is at least $2$ and no $\eta$ can improve the estimate in Conjecture \ref{ConjectureControllingLichnerowicz}.
\end{example}

The \textit{proof of Theorem \ref{TheoremThirdBettiNumber}} is analogous to the proof of Theorem \ref{GeneralRigidityTheorem}. Instead of Theorem \ref{EpsilonImprovement} use Theorem \ref{ThorpeTrick36} to show that for every harmonic $3$-form $\omega$ there exists $\eta \in \Lambda^4T^*M$ such that $0 \leq \mathcal{A}_{\omega} + \hat{\eta} \leq 2 | \omega |^2 \id_{\Lambda^2TM}.$ Hence, Theorem \ref{SharpnessCriterion} applies with $k=\frac{9}{2}$ and an application of the Bochner technique as in the proof of Theorem \ref{GeneralRigidityTheorem} yields the result. $\hfill \Box$\vspace{2mm}

\textit{Proof of Theorem \ref{SixDimHomologySpheres}}. Suppose that the oriented $6$-dimensional Riemannian manifold $(M,g)$ has $4$-positive curvature operator. By Theorem \ref{TheoremThirdBettiNumber}, the third Betti number vanishes. By \cite[Theorem A]{PetersenWinkNewCurvatureConditionsBochner}, the second Betti number vanishes. The first Betti number vanishes since $(M,g)$ has positive Ricci curvature and hence Bochner's original work \cite{BochnerVectorFieldsAndRic} applies. As $(M,g)$ is orientable, Poincar\'e duality applies. $\hfill \Box$


\begin{thebibliography}{NPW23}

\bibitem[Ber61]{BergerTwoFormsVanishCurvOperator}
Marcel Berger, \emph{Sur les vari\'{e}t\'{e}s \`a op\'{e}rateur de courbure positif}, C. R. Acad. Sci. Paris \textbf{253} (1961), 2832--2834.

\bibitem[BM18]{BettiolMendesStronglyPositiveCurvature}
Renato~G. Bettiol and Ricardo A.~E. Mendes, \emph{Strongly positive curvature}, Ann. Global Anal. Geom. \textbf{53} (2018), no.~3, 287--309.

\bibitem[Boc46]{BochnerVectorFieldsAndRic}
S.~Bochner, \emph{Vector fields and {R}icci curvature}, Bull. Amer. Math. Soc. \textbf{52} (1946), 776--797.

\bibitem[Bre08]{BrendleConvergenceInHigherDimensions}
Simon Brendle, \emph{A general convergence result for the {R}icci flow in higher dimensions}, Duke Math. J. \textbf{145} (2008), no.~3, 585--601.

\bibitem[Bre19]{BrendleRFwithSurgeryPIC}
\bysame, \emph{Ricci flow with surgery on manifolds with positive isotropic curvature}, Ann. of Math. (2) \textbf{190} (2019), no.~2, 465--559.

\bibitem[Bry06]{BryantGeometryAlmostComplexSixManifolds}
Robert~L. Bryant, \emph{On the geometry of almost complex 6-manifolds}, Asian J. Math. \textbf{10} (2006), no.~3, 561--605.

\bibitem[BS09]{BrendleSchoenSphereTheorem}
Simon Brendle and Richard Schoen, \emph{{Manifolds with 1/4-pinched curvature are space forms}}, J. Amer. Math. Soc. \textbf{22} (2009), no.~1, 287--307.

\bibitem[BW08]{BW2}
Christoph B\"ohm and Burkhard Wilking, \emph{Manifolds with positive curvature operators are space forms}, Ann. of Math. (2) \textbf{167} (2008), no.~3, 1079--1097.

\bibitem[CGT23]{CaoGurskyTranNishikawaConjecture}
Xiaodong Cao, Matthew~J. Gursky, and Hung Tran, \emph{Curvature of the second kind and a conjecture of {N}ishikawa}, Comment. Math. Helv. \textbf{98} (2023), no.~1, 195--216.

\bibitem[Che91]{ChenQuarterPinching}
Haiwen Chen, \emph{Pointwise {$\frac14$}-pinched {$4$}-manifolds}, Ann. Global Anal. Geom. \textbf{9} (1991), no.~2, 161--176.

\bibitem[Che26]{ChenPIC}
Zhengnan Chen, \emph{{Manifolds with positive isotropic curvature of dimension at least nine}}, https://arxiv.org/abs/2410.21078 (2026).

\bibitem[Cho26]{ChoPIC}
Jae~Ho Cho, \emph{{Pinching cones for positive isotropic curvature in dimensions seven and eight}}, https://arxiv.org/abs/2608.26598 (2026).

\bibitem[CTZ12]{ChenTangZhuClassFourPIC}
Bing-Long Chen, Siu-Hung Tang, and Xi-Ping Zhu, \emph{Complete classification of compact four-manifolds with positive isotropic curvature}, J. Differential Geom. \textbf{91} (2012), no.~1, 41--80.

\bibitem[CW26]{ChowWangMinimalTwoSpheresPIC}
Tsz-Kiu~Aaron Chow and Yipeng Wang, \emph{{Minimal Two-Spheres and Manifolds with Positive Isotropic Curvature}}, https://arxiv.org/abs/2609.06910 (2026).

\bibitem[CZ06]{ChenZhuRFwithSurgeryFourPIC}
Bing-Long Chen and Xi-Ping Zhu, \emph{Ricci flow with surgery on four-manifolds with positive isotropic curvature}, J. Differential Geom. \textbf{74} (2006), no.~2, 177--264.

\bibitem[Gal81]{GallotSobolevEstimates}
Sylvestre Gallot, \emph{Estim\'{e}es de {S}obolev quantitatives sur les vari\'{e}t\'{e}s riemanniennes et applications}, C. R. Acad. Sci. Paris S\'{e}r. I Math. \textbf{292} (1981), no.~6, 375--377.

\bibitem[GM75]{GallotMeyerCurvOperatorAndForms}
S.~Gallot and D.~Meyer, \emph{Op\'{e}rateur de courbure et laplacien des formes diff\'{e}rentielles d'une vari\'{e}t\'{e} riemannienne}, J. Math. Pures Appl. (9) \textbf{54} (1975), no.~3, 259--284.

\bibitem[GVZ11]{GroveVerdianiZillerExoticT1S4Pos}
Karsten Grove, Luigi Verdiani, and Wolfgang Ziller, \emph{An exotic {$T_1\mathbb{S}^4$} with positive curvature}, Geom. Funct. Anal. \textbf{21} (2011), no.~3, 499--524.

\bibitem[Ham82]{Hamilton3DimRF}
Richard~S. Hamilton, \emph{{Three-manifolds with positive Ricci curvature}}, J. Differential Geom. \textbf{17} (1982), 255--306.

\bibitem[Ham86]{Hamilton4DimRFposCurvOp}
\bysame, \emph{{Four-manifolds with positive curvature operator}}, J. Differential Geom. \textbf{24} (1986), 153--179.

\bibitem[Ham97]{HamiltonFourPIC}
\bysame, \emph{Four-manifolds with positive isotropic curvature}, Comm. Anal. Geom. \textbf{5} (1997), no.~1, 1--92.

\bibitem[Hit00]{HitchinThreeFormsSixDim}
Nigel Hitchin, \emph{The geometry of three-forms in six dimensions}, J. Differential Geom. \textbf{55} (2000), no.~3, 547--576.

\bibitem[Hua25]{HuangManifoldsPIC}
Hong Huang, \emph{{Compact manifolds of dimension $n \geq 12$ with positive isotropic curvature}}, arXiv:1909.12265 (2025).

\bibitem[Li80]{LiSobolevConstant}
Peter Li, \emph{On the {S}obolev constant and the {$p$}-spectrum of a compact {R}iemannian manifold}, Ann. Sci. \'{E}cole Norm. Sup. (4) \textbf{13} (1980), no.~4, 451--468.

\bibitem[Li25]{LiSphereTheoremsSecondKind}
Xiaolong Li, \emph{New sphere theorems under curvature operator of the second kind}, J. Lond. Math. Soc. (2) \textbf{112} (2025), no.~5, Paper No. e70356, 30.

\bibitem[Mey71]{DMeyerCurvOpPos}
Daniel Meyer, \emph{Sur les vari\'{e}t\'{e}s riemanniennes \`a op\'{e}rateur de courbure positif}, C. R. Acad. Sci. Paris S\'{e}r. A-B \textbf{272} (1971), A482--A485.

\bibitem[MM88]{MicallefMoorePIC}
Mario~J. Micallef and John~Douglas Moore, \emph{Minimal two-spheres and the topology of manifolds with positive curvature on totally isotropic two-planes}, Ann. of Math. (2) \textbf{127} (1988), no.~1, 199--227.

\bibitem[MW93]{MicallefWangNIC}
Mario~J. Micallef and McKenzie~Y. Wang, \emph{Metrics with nonnegative isotropic curvature}, Duke Math. J. \textbf{72} (1993), no.~3, 649--672.

\bibitem[NPW23]{NienhausPetersenWinkNewCurvatureOperatorSecondKind}
Jan Nienhaus, Peter Petersen, and Matthias Wink, \emph{Betti numbers and the curvature operator of the second kind}, J. Lond. Math. Soc. (2) \textbf{108} (2023), no.~4, 1642--1668.

\bibitem[Pet16]{PetersenRiemGeom}
Peter Petersen, \emph{{Riemannian Geometry}}, third ed., Graduate Texts in Mathematics, vol. 171, Springer, 2016.

\bibitem[Poo80]{PoorHolonomyProofPosCurvOperatorThm}
W.~A. Poor, \emph{A holonomy proof of the positive curvature operator theorem}, Proc. Amer. Math. Soc. \textbf{79} (1980), no.~3, 454--456.

\bibitem[P{\"u}t99]{PuettmannOptimalPinchingConstants}
Thomas P{\"u}ttmann, \emph{Optimal pinching constants of odd-dimensional homogeneous spaces}, Invent. Math. \textbf{138} (1999), no.~3, 631--684.

\bibitem[PW21a]{PetersenWinkNewCurvatureConditionsBochner}
Peter Petersen and Matthias Wink, \emph{New curvature conditions for the {B}ochner technique}, Invent. Math. \textbf{224} (2021), no.~1, 33--54.

\bibitem[PW21b]{PetersenWinkHodgeNumbers}
\bysame, \emph{Vanishing and estimation results for {H}odge numbers}, J. Reine Angew. Math. \textbf{780} (2021), 197--219.

\bibitem[Rei07]{ReichelDissertation}
Wilfried Reichel, \emph{{\"Uber die Trilinearen Alternierenden Formen in $6$ und $7$ Ver\"anderlichen}}, Dissertation, Greifswald (1907).

\bibitem[Roc70]{RockafellarConvexAnalysis}
R.~Tyrrell Rockafellar, \emph{Convex analysis}, Princeton Mathematical Series, vol. No. 28, Princeton University Press, Princeton, NJ, 1970.

\bibitem[Tho72]{ThorpeCurvatureTensorPosCurvedFourMf}
John~A. Thorpe, \emph{On the curvature tensor of a positively curved {$4$}-manifold}, Proceedings of the {T}hirteenth {B}iennial {S}eminar of the {C}anadian {M}athematical {C}ongress ({D}alhousie {U}niv., {H}alifax, {N}.{S}., 1971), {V}ol. 2, Canad. Math. Congr., Montreal, QC, 1972, pp.~156--159.

\end{thebibliography}

\end{document}